\documentclass[11pt,a4paper,twoside,english]{article}

\usepackage{amssymb,amsmath,amsthm}
\usepackage{graphicx}
\usepackage{babel}
\usepackage{color}
\usepackage[round]{natbib}

\def\bs{\boldsymbol}

\def\tilde{\widetilde}

\def\DTV{\mathop{d_{\rm TV}}}

\def\esssup{\mathop{\mathrm{ess\,sup}}}

\def\AA{\mathcal{A}}
\def\LL{\mathcal{L}}
\def\XX{\mathcal{X}}

\def\Bin{\mathop{\mathrm{Bin}}\nolimits}
\def\Hyp{\mathop{\mathrm{Hyp}}\nolimits}
\def\Poiss{\mathop{\mathrm{Poiss}}\nolimits}

\def\Ex{\mathop{\mathrm{I\!E}}\nolimits}
\def\Pr{\mathop{\mathrm{I\!P}}\nolimits}

\def\Var{\mathop{\mathrm{Var}}\nolimits}

\newcommand{\R}{\mathbb{R}}

\newtheoremstyle{localthm}
	{5pt} 
	{5pt} 
	{\sl} 
	{} 
	{\bf} 
	{{\rm.}} 
	{.7em} 
	{} 

\theoremstyle{localthm}
\newtheorem{Theorem}{Theorem}
\newtheorem{Proposition}[Theorem]{Proposition}
\newtheorem{Corollary}[Theorem]{Corollary}
\newtheorem{Lemma}[Theorem]{Lemma}

\newtheoremstyle{localrem}
	{5pt} 
	{5pt} 
	{\rm} 
	{} 
	{\bf} 
	{{\rm.}} 
	{.7em} 
	{} 

\theoremstyle{localrem}
\newtheorem{Example}[Theorem]{Example}
\begin{document}

\addtolength{\baselineskip}{0.32\baselineskip}

\title{\bf Bounding distributional errors via density ratios}
\author{Lutz D\"{u}mbgen (University of Bern){\footnote{Research supported by Swiss National Science Foundation}}, \\
	Richard J.\ Samworth (University of Cambridge)\footnote{Research supported by an Engineering and Physical Sciences 
Research Council fellowship}
	\ and\\
	Jon A.\ Wellner (University of Washington, Seattle){\footnote{Research supported in part by NSF Grant DMS-1566514 and NI-AID Grant 2R01 AI291968-04}}
}
\date{February 4, 2020}
\maketitle

\begin{abstract}
We present some new and explicit error bounds for the approximation of distributions. 
The approximation error is quantified by the maximal density ratio of the distribution 
$Q$ to be approximated and its proxy $P$. This non-symmetric measure is more 
informative than and implies bounds for the total variation distance.

Explicit approximation problems include, among others, hypergeometric by 
binomial distributions, binomial by Poisson distributions, and beta by gamma distributions.
In many cases we provide both upper and (matching) lower bounds.
\end{abstract}

\paragraph{Key words:} Binomial distribution, hypergeometric distribution, Poisson approximation, relative errors, total variation distance.

\section{Introduction}
\label{sec:Introduction}

The aim of this work is to provide new inequalities for the approximation of probability distributions. A traditional measure of discrepancy between distributions $P, Q$ on a space $(\XX,\AA)$ is their total variation distance
\[
	\DTV(Q,P) \ := \ \sup_{A \in \AA} \, \bigl| Q(A) - P(A) \bigr| .
\]
Alternatively we consider the maximal ratio
\[
	\rho(Q,P) \ := \ \sup_{A \in \AA} \, \frac{Q(A)}{P(A)},
\]
with the conventions $0/0 := 0$ and $a/0 := \infty$ for $a > 0$. Obviously $\rho(Q,P) \ge 1$ because $Q(\XX) = P(\XX) = 1$. While $\DTV(\cdot,\cdot)$ is a standard and strong metric on the space of all probability measures on $(\XX,\AA)$, the maximal ratio $\rho(Q,P)$ is particularly important in situations in which a distribution $Q$ is approximated by a distribution $P$. When $\rho(Q,P) < \infty$, we know that
\[
	Q(A) \ \le \ \rho(Q,P) P(A)
\]
for arbitrary events $A$, no matter how small $P(A)$ is, whereas total variation distance gives only the additive bounds $P(A) \pm \DTV(Q,P)$.

Explicit values or bounds for $\rho(Q,P)$ are obtained via density ratios. From now on let $P$ and $Q$ have densities $f$ and $g$, respectively, with respect to some measure $\mu$ on $(\XX,\AA)$. Then
\begin{equation}
\label{eq:rho.via.LR}
	\rho(Q,P) \ = \ \esssup_{x \in \XX} \, \frac{g(x)}{f(x)} .
\end{equation}
The ratio measure $\rho(Q,P)$ plays an important role in acceptance-rejection sampling \citep{vonNeumann_1951}: Suppose that $\rho(Q,P) \le C < \infty$. Let $X_1, X_2, X_3, \ldots$ and $U_1, U_2, U_3, \ldots$ be independent random variables where $X_i \sim P$ and $U_i \sim \mathrm{Unif}[0,1]$. Now let $\tau_1 < \tau_2 < \tau_3 < \cdots$ denote all indices $i \in \mathbb{N}$ such that $U_i \le C^{-1} g(X_i)/f(X_i)$. Then the random variables $Y_j := X_{\tau_j}$ and $W_j := \tau_j - \tau_{j-1}$ ($j \in \mathbb{N}$, $\tau_0 := 0$) are independent with $Y_j \sim Q$ and $W_j \sim \mathrm{Geom}(1/C)$.

As soon as we have a finite bound for $\rho(Q,P)$, we can bound total variation distance or other measures of discrepancy. The general result is as follows:

\begin{Proposition}
\label{prop:rho.and.divergences}
Suppose that $g/f \le \rho$ for some number $\rho \in [1,\infty)$.

\noindent
\textbf{(a)} For any non-decreasing function $\psi : [0,\infty) \to \R$ with $\psi(1) = 0$,
\[
	\int \psi(g/f) \, dQ \ \le \ Q(\{g > f\}) \psi(\rho) .
\]

\noindent
\textbf{(b)} For any convex function $\psi : [0,\infty) \to \R$,
\[
	\int \psi(g/f) \, dP \ \le \ \psi(0) + \frac{\psi(\rho) - \psi(0)}{\rho} .
\]

\noindent
Both inequalities are equalities if $g/f$ takes only values in $\{0,\rho\}$.
\end{Proposition} 

Under the assumptions of Proposition~\ref{prop:rho.and.divergences}, the following inequalities hold true, with equality in case of $g/f \in \{0,\rho\}$:\\[0.5ex]
\textbf{Total variation}: With $\psi(t) := (1 - t^{-1})_+$, part~(a) leads to
\begin{equation}
\label{ineq:rho.and.DTV}
	\DTV(Q,P) \ \le \ Q(\{g > f\}) (1 - \rho^{-1}) .
\end{equation}
\textbf{Kullback-Leibler divergence}: With $\psi(t) := \log t$, part~(a) yields
\[
	\int \log(g/f) \, dQ \ \le \ Q(\{g > f\}) \log \rho .
\]
\textbf{Hellinger distance}: With $\psi(t) := 2^{-1} \bigl( \sqrt{t} - 1 \bigr)^2$, part~(b) leads to
\[
	\frac{1}{2} \int \bigl( \sqrt{f} - \sqrt{g} \bigr)^2 \, d\mu
	\ \le \ 1 - \rho^{-1/2} .
\]
\textbf{Pearson $\chi^2$ divergence}: With $\psi(t) := (t - 1)^2$, part~(b) yields
\[
	\int (g/f - 1)^2 \, dP
	\ \le \ \rho - 1 .
\]

Inequality~\eqref{ineq:rho.and.DTV} implies that $\DTV(Q,P) \le 1 - \rho(Q,P)^{-1}$, and the latter quantity is easily seen to be the \textsl{mixture index of fit} introduced by \cite{Rudas_Clogg_Lindsay_1994},
\begin{align*}
	\pi^*(P,Q) \
	&:= \ \min \bigl\{ \pi \in [0,1] : P = (1 - \pi) Q + \pi R
		\ \text{for some distribution} \ R \bigr\} \\
	&\phantom{:}= \ \min \bigl\{ \pi \in [0,1] :
		P \ge (1 - \pi) Q \ \text{on} \ \AA \bigr\} .
\end{align*}

The remainder of this paper is organized as follows: In Section~\ref{sec:HypBin} we present an explicit inequality for $\rho(Q,P)$ with $Q$ being a hypergeometric and $P$ being an approximating binomial distribution. Our result complements results of \cite{Diaconis_Freedman_1980}, \cite{Ehm_1991} and \cite{Holmes_2004} for $\DTV(Q,P)$.

In Section~\ref{sec:BinPoiss} we first consider the case of $Q$ being a binomial distribution and $P$ being the Poisson distribution with the same mean. The corresponding ratio measure $\rho(Q,P)$ has been analyzed previously by \cite{Christensen_Fischer_Kvols_1995} and \cite{Antonelli_Regoli_2005}. Our new explicit bounds bridge the gap between these two works. As a by-product we obtain explicit bounds for $\DTV(Q,P)$ which are comparable to well-known bounds from the literature. All these bounds carry over to multinomial distributions, to be approximated by a product of Poisson distributions. In particular, we improve and generalize approximation bounds by \cite{Diaconis_Freedman_1987}. Indeed, at several places we use sufficiency arguments similarly to the latter authors to reduce multivariate approximation problems to univariate ones. Section~\ref{sec:Further.examples} presents several further examples, most of which are based on approximating beta by gamma distributions.

Most proofs are deferred to Section~\ref{sec:Proofs}. In particular, we provide a slightly strengthened version of the Stirling--Robbins approximation of factorials \citep{Robbins_1955} and some properties of the log-gamma function. This part is potentially of independent interest. As notation used throughout, we write $[a]_0 := 1$ and $[a]_m := \prod_{i=0}^{m-1} (a - i)$ for real numbers $a$ and integers $m \ge 1$.

\section{Binomial approximation of hypergeometric distributions}
\label{sec:HypBin}

\paragraph{Sampling from a finite population.}
First we revisit a result of \cite{Freedman_1977} concerning sampling with and without replacement. For integers $1 \le n \le N$ let $\XX = \{1,\ldots,N\}^n$, the set of all samples of size $n$ drawn with replacement from $\{1,\ldots,N\}$. The uniform distribution $P$ on $\XX$ has weights
\[
	P(\{x\}) = 1/N^n
\]
for $x = (x_1,\ldots,x_n) \in \XX$. When sampling without replacement, we consider the set $\XX_*$ of all $x$ with all components different, and the distribution $Q$ with weights
\[
	Q(\{x\}) = \begin{cases}
		1/[N]_n & \text{if} \ x \in \XX_* , \\
		0       & \text{if} \ x \in \XX\setminus \XX_* .
	\end{cases}
\]
Consequently, $dQ/dP = N^n/[N]_n$ on $\XX_*$ and $dQ/dP = 0$ on $\XX\setminus\XX_*$, so Proposition~\ref{prop:rho.and.divergences}~(a) with $\psi(t) := (1 - t^{-1})_+$ implies that
\begin{equation}
\label{eq:rho.DTV.sampling}
	\rho(Q,P) = N^n/[N]_n
	\quad\text{and}\quad
	\DTV(Q,P) = 1 - \rho(Q,P)^{-1} = 1 - [N]_n/N^n .
\end{equation}
\cite{Freedman_1977} showed that
\begin{equation}
\label{ineq:Freedman.sampling}
	1 - \exp \Bigl( - \frac{n(n-1)}{2 N} \Bigr)
	\le \DTV(Q,P) \le \frac{n(n-1)}{2N} .
\end{equation}
Here are two new bounds for $\rho(Q,P)$ which we will prove in Section~\ref{sec:Proofs}. The lower bound in the following display follows from Freedman's proof of the lower bound in (\ref{ineq:Freedman.sampling}), while the upper bound is new.
\begin{equation}
\label{ineq:rho.sampling}
	\frac{n(n-1)}{2N}
	\ \le \ \log \rho(Q,P)
	\ \le \ - \frac{n}{2} \log \Bigl( 1 - \frac{n-1}{N} \Bigr) .
\end{equation}
From \eqref{eq:rho.DTV.sampling} and \eqref{ineq:Freedman.sampling} one would get the upper bound $- \log \bigl( 1 - n(n-1)/(2N) \bigr)$ with the convention that $\log(t) := -\infty$ for $t \le 0$. For $n = 2$ this coincides with the upper bound in \eqref{ineq:rho.sampling}, for $n \ge 3$ it is strictly larger.

\paragraph{Hypergeometric and binomial distributions.}
Now recall the definition of the hypergeometric distribution: Consider an urn with $N$ balls, $L$ of them being black and $N - L$ being white. Now we draw $n$ balls at random and define $X$ to be the number of black balls in this sample. When sampling with replacement, $X$ has the binomial distribution $\Bin(n,L/N)$, and when sampling without replacement ($n \le N$), $X$ has the hypergeometric distribution $\Hyp(N,L,n)$. Intuitively one would guess that the difference between $\Bin(n,L/N)$ and $\Hyp(N,L,n)$ is small when $n \ll N$. Note that when \nocite{Freedman_1977}{Freedman's (1977)} result is applied to a particular function, e.g.\ the number of black balls, the resulting bound is suboptimal because it involves $n(n-1)/N$ rather than $n/N$. Indeed, \cite{Diaconis_Freedman_1980} showed that
\[
	\DTV \bigl( \Bin(n,L/N), \Hyp(N,L,n) \bigr)
	\ \le \ 2 \frac{n}{N} .
\]
Stronger bounds have been obtained by means of the Chen--Stein method. \cite{Ehm_1991} showed that with $p := L/N$,
\begin{align}
\nonumber
	\DTV \bigl( & \Hyp (N,L,n),\Bin (n,L/N) \bigr) \\
\label{ineq:Ehm}
	&\le \ \frac{n}{n+1} \left( 1 - p^{n+1} - (1 - p)^{n+1} \right) \frac{n-1}{N-1}
	\quad\text{for} \ 1 \le n \le \min\{L,N-L\} ,
\end{align}
while \cite{Holmes_2004} proved that
\begin{equation}
\label{ineq:Holmes}
	\DTV \bigl( \Hyp (N,L,n),\Bin (n,L/N) \bigr)
	\ \le \ \frac{n-1}{N-1} .
\end{equation}

Our first main result shows that for fixed parameters $N$ and $n \le N/2+1$, the ratio measure $\rho \bigl( \Hyp(N,L,n), \Bin(n,L/N) \bigr)$ is maximized by $L = 1$ (and $L = N-1$):

\begin{Theorem}
\label{thm:HypBin}
For integers $N,L,n$ with $1 \le n \le N$, $n-1 \le N/2$ and $L \in \{0,1,\ldots,N\}$,
\begin{align*}
	\rho \bigl( \Hyp(N,L,n), \Bin(n,L/N) \bigr) \
	&\le \ \rho \bigl( \Hyp(N,1,n), \Bin(n,1/N) \bigr) \\
	&= \ \Bigl( 1 - \frac{1}{N} \Bigr)^{-(n-1)} \\
	&\le \ \Bigl( 1 - \frac{n-1}{N} \Bigr)^{-1} .
\end{align*}
Moreover,
\begin{align*}
	\DTV \bigl( \Hyp(N,L,n), \Bin(n,L/N) \bigr) \ 
	&\le \left ( 1 - \frac{[L]_n}{[N]_n} - \frac{[N-L]_n}{[N]_n} \right)
		\left( 1 - \Bigl(1 - \frac{1}{N}\Bigr)^{n-1} \right )\\ 
	&\le \left ( 1 - \frac{[L]_n}{[N]_n} - \frac{[N-L]_n}{[N]_n} \right )\frac{n-1}{N} .
\end{align*}
\end{Theorem}

\paragraph{Remarks.} Note that our bounds for $\DTV \bigl(\Hyp(N,L,n), \Bin(n,L/N) \bigr)$ are slightly better than the bound \eqref{ineq:Holmes} of \cite{Holmes_2004}. If we fix $n$ and let $L,N \to \infty$ such that $L/N \to p \in (0,1)$, then our bounds are equal to
\[
	(1 + o(1)) (1 - p^n - q^n) \frac{n-1}{N}
\]
and thus similar to the bound \eqref{ineq:Ehm} of \cite{Ehm_1991}. If we fix $L$ and let $n,N \to \infty$ such that $n/N \to \gamma \in (0,1)$, then our two bounds converge to
\[
	(1 - \gamma^L) (1 - e^{-\gamma}) \ \le \ (1 - \gamma^L) \gamma ,
\]
whereas the upper bound in \eqref{ineq:Holmes} tends to $\gamma$, and \eqref{ineq:Ehm} is not applicable.

\section{Poisson approximations}
\label{sec:BinPoiss}

\subsection{Binomial distributions}

It is well-known that for $n \in \mathbb{N}$ and $p \in [0,1]$, the binomial distribution $\Bin(n,p)$ may be approximated by the Poisson distribution $\Poiss(np)$ if $p$ is small. Explicit bounds for the approximation error have been developed in the more general setting of sums of independent but not necessarily identically distributed Bernoulli random variables by various authors. \cite{Hodges_LeCam_1960} introduced a coupling method which was refined by \cite{Serfling_1975} and implies the inequality
\[
	\DTV \bigl( \Bin(n,p),\Poiss(np) \bigr) \ \le \ np (1 - e^{-p}) \ \le \ n p^2 .
\]
By direct calculations involving density ratios, \cite{Reiss_1993} showed that
\[
	\DTV \bigl( \Bin(n,p),\Poiss(np) \bigr) \ \le \ p .
\]
Finally, by means of the Chen--Stein method, \cite{Barbour_Hall_1984} derived the remarkable bound
\begin{equation}
\label{ineq:Barbour-Hall}
	\DTV \bigl( \Bin(n,p), \Poiss(np) \bigr)
	\ \le \ (1 - e^{-np}) p .
\end{equation}

Concerning the ratio measure $\rho \bigl( \Bin(n,p), \Poiss(np) \bigr)$, \cite{Christensen_Fischer_Kvols_1995} showed that
\[
	\Lambda(p) \ := \ \max_{n \ge 1} \, \log \rho \bigl( \Bin(n,p), \Poiss(n,p) \bigr)
\]
is a convex, piecewise linear function of $p \in [0,1)$ with $\lim_{p\to 1} \Lambda(p) = \infty$ and
\begin{equation}
\label{eq:CFK1}
	\Lambda(p) \ = \ p
	\quad \text{for} \ 0 \le p \le \log(2) .
\end{equation}
A close inspection of their proof reveals that $\Lambda(p)$ is the maximum of the log-ratio measure $\log \rho \bigl( \Bin(n,p), \Poiss(n,p) \bigr)$ over all integers $n \le 1/(1 - p)$, so the bound $\Lambda(p)$ is probably rather conservative for large sample sizes $n$. Indeed, it follows from the results of \cite{Antonelli_Regoli_2005} that for any fixed $p \in (0,1)$,
\begin{equation}
\label{eq:Antonelli_Regoli}
	\lim_{n \to \infty} \log \rho \bigl( \Bin(n,p), \Poiss(np) \bigr)
	\ = \ - \log(1 - p)/2
\end{equation}
which is substantially smaller than $\Lambda(p)$, at least for small values $p$. By means of elementary calculations and an appropriate version of Stirling's formula, we shall prove the following bounds:

\begin{Theorem}
\label{thm:BinPoiss1}
For arbitrary $n \in \mathbb{N}$,
\[
	\Lambda_n(p) \ := \ \log \rho \bigl( \Bin(n,p), \Poiss(np) \bigr)
\]
is a continuous and strictly increasing function of $p \in [0,1)$, satisfying $\Lambda_n(0) = 0$ and
\[
	\Lambda_n(p) 
	\ < \ \begin{cases}
		- \log(1 - p) \\[0.5ex]
		- \log(1 - \lceil np\rceil/n) / 2
	\end{cases}
\]
for $0 < p < 1$. More precisely, with $k := \lceil np\rceil$,
\begin{equation}
\label{ineq:BinPoiss1}
	\Lambda_n(p) + \log(1 - p)/2
	\begin{cases}
		\displaystyle
		< \ - \frac{k - 1}{12 n (n - k + 1)}
			+ \frac{1}{8(n - k) + 6} , \\[2ex]
		\displaystyle
		> \ - \frac{k-1}{12 n (n - k + 1)}
			- \frac{1}{12 (n - k)(n - k + 1)} .
	\end{cases}
\end{equation}
\end{Theorem}

\paragraph{Remarks.}
Since $P(\{0\}) = e^{-np} \ge Q(\{0\}) = (1 - p)^n$, the first two upper bounds of Theorem~\ref{thm:BinPoiss1} and Proposition~\ref{prop:rho.and.divergences}~(a) lead to the inequalities
\[
	\DTV \bigl( \Bin(n,p), \Poiss(np) \bigr)
	\ < \ \bigl( 1 - (1 - p)^n \bigr) \cdot \begin{cases}
		p , \\
		\displaystyle
			1 - \sqrt{1 - \frac{\lceil np\rceil}{n}}
			\ \le \ \frac{\lceil np\rceil/n}{2 - \lceil np\rceil/n} ;
	\end{cases}
\]
see inequality~\eqref{eq:sqrt.delta} in Section~\ref{sec:Proofs}. For fixed $\lambda > 0$, the bound in \eqref{ineq:Barbour-Hall} may be rephrased as $n \DTV \bigl( \Bin(n,\lambda/n), \Poiss(\lambda) \bigr) \le (1 - e^{-\lambda}) \lambda$. Our bounds imply that
\[
	\limsup_{n\to \infty} \, \DTV \bigl( \Bin(n,\lambda/n), \Poiss(\lambda) \bigr)
	\ \le \ (1 - e^{-\lambda}) \min \bigl\{ \lambda, \lceil\lambda\rceil/2 \bigr\} ,
\]
and $\lceil\lambda\rceil/2 < \lambda$ for $\lambda > 1/2$. The refined inequalities imply that for any fixed $p_o \in (0,1)$,
\[
	\log \rho \bigl( \Bin(n,p), \Poiss(np) \bigr)
	\ = \ - \log(1 - p)/2 + O(n^{-1})
	\quad\text{uniformly in} \ p \le p_o .
\]

The proof of Theorem~\ref{thm:BinPoiss1} reveals that $\Lambda_n(p) = \log \rho \bigl( \Bin(n,p), \Poiss(np) \bigr)$ is concave in $p \in \bigl[ (k-1)/n, k/n \bigr]$ for each $k \in \{1,\ldots,n\}$. Figure~\ref{fig:BinPoiss40} illustrates this for $n = 40$. In the left panel one sees $\Lambda_n(p)$ (black) together with $\Lambda(p)$ (black dashed) and the simple upper bounds $- \log(1 - p)$ (green) and $- \log(1 - \lceil np\rceil/n)/2$ (blue). The right panel shows the quantities $\Lambda_n(p) + \log(1 - p)/2$ (black), i.e.\ the difference of $\Lambda_n(p)$ and the asymptotic bound $- \log(1 - p)/2$ of \cite{Antonelli_Regoli_2005}, together with the upper bound $- \log(1 - \lceil np\rceil/n)/2 + \log(1 - p)/2$ (blue) and the two bounds in \eqref{ineq:BinPoiss1} (red and orange).

\begin{figure}
\includegraphics[width=0.49\textwidth]{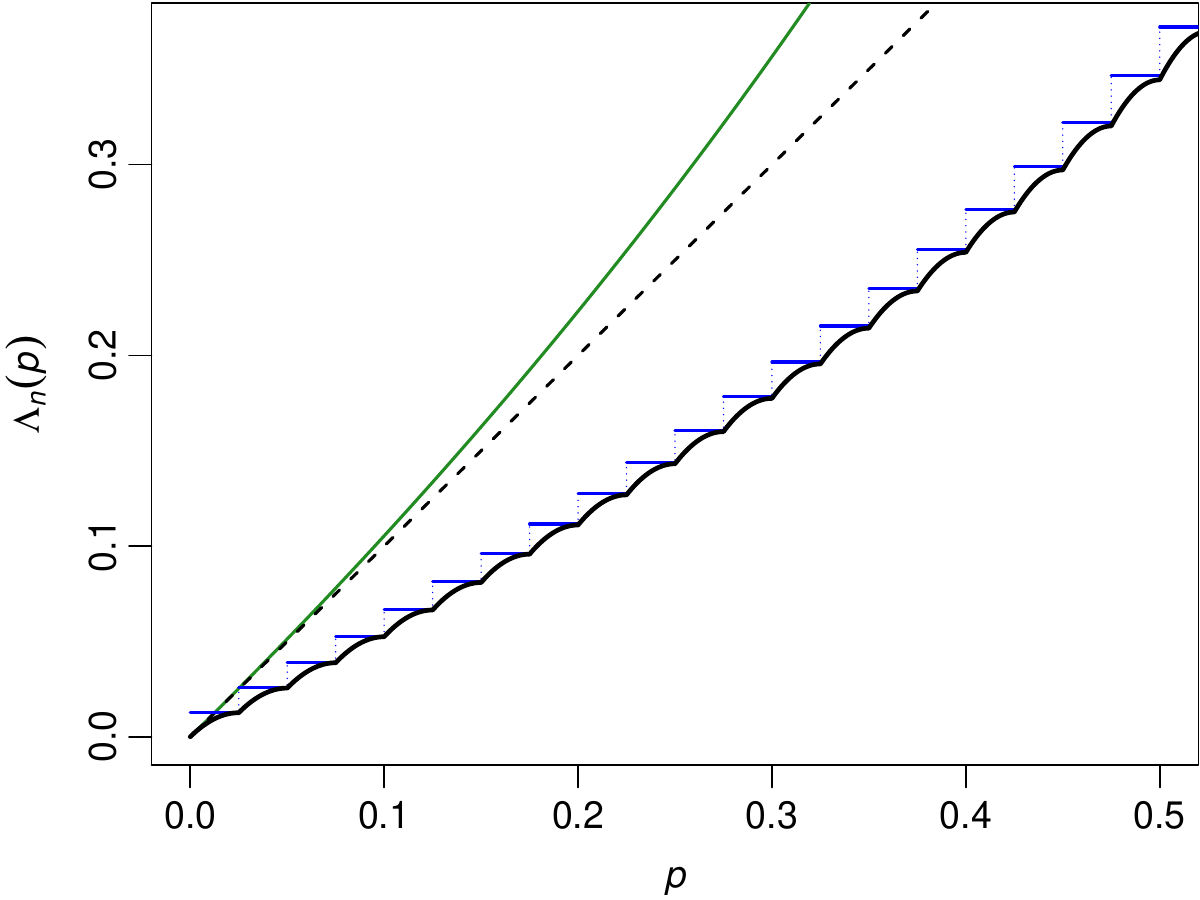}
\hfill
\includegraphics[width=0.49\textwidth]{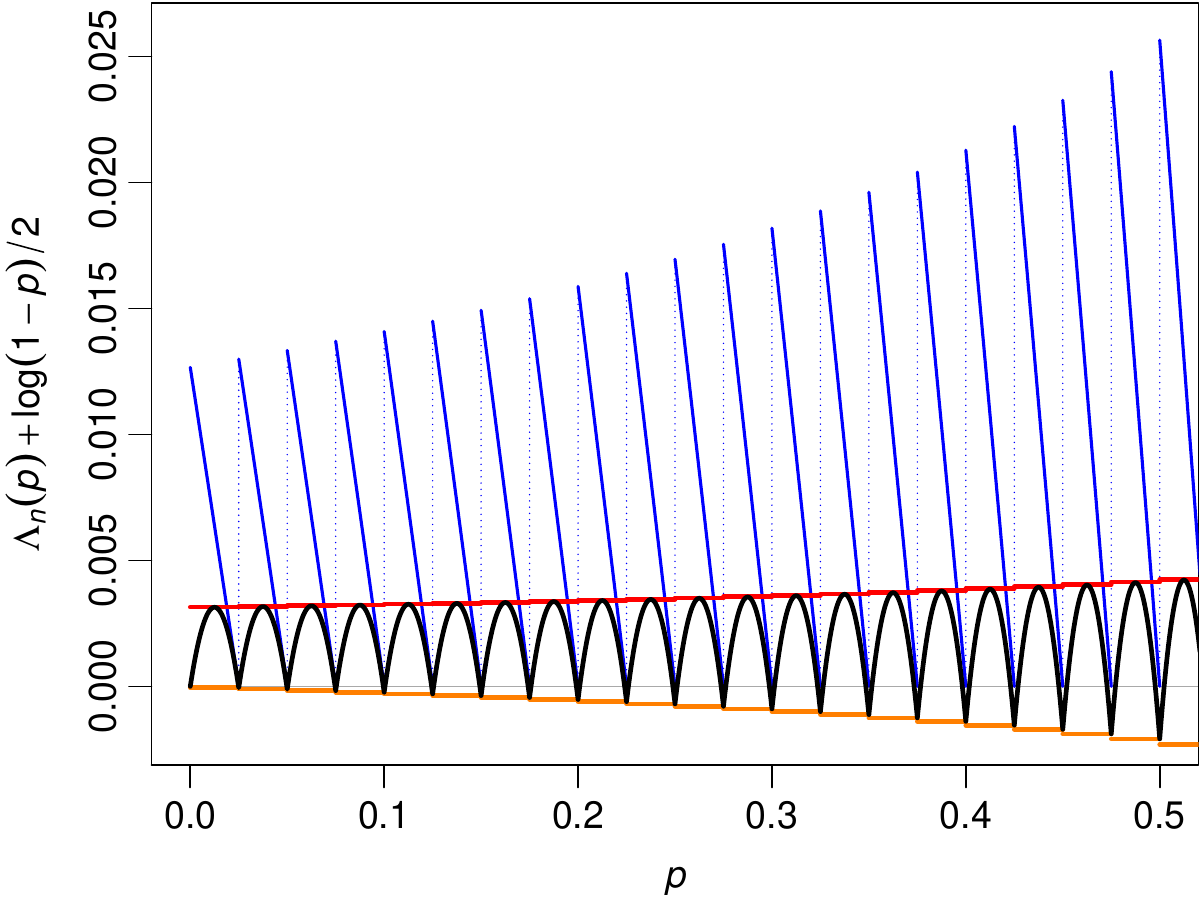}
\caption{Comparing $\Bin(40,p)$ with $\Poiss(40\,p)$.}
\label{fig:BinPoiss40}
\end{figure}

\paragraph{Poisson binomial distributions.}
The distribution $\Bin(n,p)$ can be replaced with the distribution $Q$ of 
$\sum_{i=1}^n Z_i$ with independent Bernoulli variables $Z_i$ with arbitrary parameters 
$p_i := \Pr(Z_i = 1) \in (0,1)$ and $\lambda := \sum_{i=1}^n p_i$ in place of $np$. \cite{Duembgen_Wellner_2019} showed that $\rho(Q, \Poiss(\lambda)) \le (1 - p_*)^{-1}$ with $p_* := \max_{1 \le i \le n} p_i$.

\subsection{Multinomial distributions and Poissonization}

\paragraph{Multinomial distributions.}
The previous bounds for the approximation of binomial by Poisson distributions imply bounds for the approximation of multinomial distributions by products of Poisson distributions. For integers $n, K \ge 1$ and parameters $p_1, \ldots, p_K > 0$ such that $p_+ := \sum_{i=1}^K p_i < 1$, let $(Y_0,Y_1,\ldots,Y_K)$ follow a multinomial distribution
\[
	\mathrm{Mult}(n; p_0, p_1, \ldots, p_K) ,
\]
where $p_0 := 1 - p_+$. Further, let $X_1, \ldots, X_K$ be independent Poisson random variables with parameters $np_1, \ldots, np_K$ respectively. Elementary calculations reveal that with $Y_+ := \sum_{i=1}^K Y_i$ and $X_+ := \sum_{i=1}^K X_i$,
\[
	\LL(Y_1,\ldots,Y_K \,|\, Y_+ = m)
	\ = \ \LL(X_1,\ldots,X_K \,|\, X_+ = m)
	\ = \ \mathrm{Mult} \Bigl( m; \frac{p_1}{p_+}, \ldots, \frac{p_K}{p_+} \Bigr)
\]
for arbitrary integers $m \ge 0$. Moreover,
\[
	Y_+ \ \sim \ \Bin(n, p_+)
	\quad\text{and}\quad
	X_+ \ \sim \ \Poiss(n p_+) .
\]
This implies that for arbitrary integers $x_1, \ldots, x_K \ge 0$ and $x_+ := \sum_{i=1}^K x_i$,
\[
	\frac{\Pr(Y_i = x_i \ \text{for} \ 1 \le i \le K)}
	     {\Pr(X_i = x_i \ \text{for} \ 1 \le i \le K)}
	\ = \ \frac{\Pr(Y_+ = x_+)}{\Pr(X_+ = x_+)} .
\]
Consequently, by \eqref{eq:rho.via.LR},
\[
	\rho \bigl( \LL(X_1,\ldots,X_K), \LL(Y_1,\ldots,Y_K) \bigr)
	\ = \ \rho \bigl( \Bin(n,p_+), \Poiss(np_+) \bigr) ,
\]
and one easily verifies that
\[
	\DTV \bigl( \LL(X_1,\ldots,X_K), \LL(Y_1,\ldots,Y_K) \bigr)
	\ = \ \DTV \bigl( \Bin(n,p_+), \Poiss(np_+) \bigr) .
\]

\paragraph{Poissonization.}
Theorem~\ref{thm:BinPoiss1} applies also to Poissonization for empirical processes: Let $X_1, X_2, X_3, \ldots$ be independent random variables with distribution $P$ on a measurable space $(\XX,\AA)$. Let $M_n$ be the random measure $\sum_{i=1}^n \delta_{X_i}$, and let $\tilde{M}_n$ be a Poisson process on $(\XX,\AA)$ with intensity measure $nP$. Then $\tilde{M}_n$ has the same distribution as $\sum_{i \le N_n} \delta_{X_i}$, where $N_n \sim \Poiss(n)$ is independent from $(X_i)_{i \ge 1}$. For a set $A_o \in \AA$ with $0 < p_o := P(A_o) < 1$, the restrictions of the random measures $M_n$ and $\tilde{M}_n$ to $A_o$ satisfy the equality
\[
	\rho \bigl( \LL(M_n\vert_{A_o}), \LL(\tilde{M}_n\vert_{A_o}) \bigr)
	\ = \ \rho \bigl( \Bin(n,p_o), \Poiss(np_o) \bigr) .
\]
Here $M_n\vert_{A_o}$ and $\tilde{M}_n\vert_{A_o}$ stand for the random measures
\[
	\{A \in \mathcal{A} : A \subseteq A_o\} \ni A \ \mapsto \ M_n(A), \tilde{M}_n(A)
\]
on $A_o$. Indeed, for arbitrary integers $m \ge 0$,
\[
	\LL \bigl( M_n\vert_{A_o} \,\big|\, M_n(A_o) = m \bigr)
	\ = \ \LL \bigl( \tilde{M}_n\vert_{A_o} \,\big|\, \tilde{M}_n(A_o) = m \bigr) ,
\]
while
\[
	M_n(A_o) \ \sim \ \Bin(n, p_o)
	\quad\text{and}\quad
	\tilde{M}_n(A_o) \ \sim \ \Poiss(n p_o) .
\]
Consequently,
\[
	\rho \bigl( \LL(M_n\vert_{A_o}), \LL(\tilde{M}_n\vert_{A_o}) \bigr)
	\ = \ \rho \bigl( \Bin(n,p_o), \Poiss(np_o) \bigr)
\]
and
\[
	\DTV \bigl( \LL(M_n\vert_{A_o}), \LL(\tilde{M}_n\vert_{A_o}) \bigr)
	\ = \ \DTV \bigl( \Bin(n,p_o), \Poiss(np_o) \bigr) .
\]

\section{Gamma approximations and more}
\label{sec:Further.examples}

In this section we present further examples of bounds for the ratio measure $\rho(Q,P)$. In all but one case, they are related to the approximation of beta by gamma distributions.

\subsection{Beta distributions}

In what follows, let $\mathrm{Beta}(a,b)$ be the beta distribution with parameters $a, b > 0$. The corresponding density is given by
\[
	\beta_{a,b}(x) \ = \ \frac{\Gamma(a+b)}{\Gamma(a) \Gamma(b)} \, x_{}^{a-1} (1 - x)_+^{b-1} ,
	\quad x > 0 ,
\]
with the gamma function $\Gamma(a) := \int_0^\infty x^{a-1} e^{-x} \, dx$. Note that we view $\mathrm{Beta}(a,b)$ as a distribution on the halfline $(0,\infty)$, because we want to approximate it by gamma distributions. Specifically, let $\mathrm{Gamma}(a,c)$ be the gamma distribution with shape parameter $a > 0$ and rate parameter (i.e.\ inverse scale parameter) $c > 0$. The corresponding density is given by
\[
	\gamma_{a,c}(x) \ = \ \frac{c^a}{\Gamma(a)} \, x_{}^{a-1} e_{}^{-cx} ,
	\quad x > 0 ,
\]
The next theorem shows that $\mathrm{Beta}(a,b)$ may be approximated by $\mathrm{Gamma}(a,c)$ for suitable rate parameters $c > 0$, provided that $b \gg \max(a,1)$.

\begin{Theorem}
\label{thm:Beta.Gamma}
(i) For arbitrary parameters $a > 0$ and $b > 1$,
\begin{align*}
	\rho \bigl( \mathrm{Beta}(a,b), \mathrm{Gamma}(a,a+b) \bigr) \
	&\le \ (1 - \delta)^{-1/2}
		\qquad\text{and} \\
	\DTV \bigl( \mathrm{Beta}(a,b), \mathrm{Gamma}(a, a+b) \bigr) \
	&\le \ 1 - (1-\delta)^{1/2} \ < \ \frac{\delta}{2 - \delta},
\end{align*}
where
\[
	\delta \ := \ \frac{a+1}{a + b} .
\]

\noindent
(ii) For $a > 0$, $b > 1$, and arbitrary $c>0$,
\[
	\rho \bigl( \mathrm{Beta}(a,b), \mathrm{Gamma}(a, c) \bigr)
	\ \ge \ \rho \bigl( \mathrm{Beta}(a,b), \mathrm{Gamma}(a, a+b-1) \bigr) .
\]
Moreover, for this opimal rate parameter $c = a+b-1$,
\begin{align*}
	\rho \bigl( \mathrm{Beta}(a,b), \mathrm{Gamma}(a, a+b-1) \bigr) \
	&\le \ (1 - \widetilde{\delta})^{-1/2}
		\qquad\text{and} \\
	\DTV \bigl( \mathrm{Beta}(a,b), \mathrm{Gamma}(a, a+b-1) \bigr) \
	&\le \ 1 - (1-\widetilde{\delta})^{1/2}
		\ < \ \frac{\widetilde{\delta}}{2 - \widetilde{\delta}},
\end{align*}
where 
\[
	\widetilde{\delta} \ := \ \frac{a}{a+b-1} \ < \ \delta .
\]
\end{Theorem}

\paragraph{Remarks.}
The rate parameter $c = a+b$ is canonical in the sense that the means of $\mathrm{Beta}(a,b)$ and $\mathrm{Gamma}(a,a+b)$ are both equal to $a/(a + b)$. But note that 
\[
	\frac{\widetilde{\delta}}{\delta}
	\ = \ \frac{a}{a+1} \cdot \frac{a+b}{a+b-1}
	\ \approx \ \frac{a}{a+1} 
\]
if $b \gg \max\{a,1\}$. Hence, $\mathrm{Gamma}(a, a+b-1)$ yields a remarkably better approximation than $\mathrm{Gamma}(a,a+b)$, unless $a$ is rather large or $b$ is close to $1$.

In the proof of Theorem~\ref{thm:Beta.Gamma} it is shown that in the special case of $a = 1$, one can show the following: For $b > 1$,
\[
	\log \rho \bigl( \mathrm{Beta}(1,b), \mathrm{Gamma}(1, b) \bigr)
	\ = \ (b - 1) \log(1 - 1/b) + 1 ,
\]
and for $b \ge 2$,
\[
	\left.\begin{array}{r}
		\log \rho \bigl( \mathrm{Beta}(1,b), \mathrm{Gamma}(1, b) \bigr) \\[1ex]
		\DTV \bigl( \mathrm{Beta}(1,b), \mathrm{Gamma}(1, b) \bigr)
	\end{array}\!\!\right\}
	\ \le \ \frac{1}{2b} + \frac{1}{4b^2} .
\]

\subsection{The L\'{e}vy--Poincar\'{e} projection problem} 

Let $\bs{U} = (U_1, U_2, \ldots , U_n)$ be uniformly distributed on the unit sphere in $\R^n$. It is well-known that $\bs{U}$ can be represented as $\bs{Z} / \|\bs{Z}\|$ where $\bs{Z} \sim N_n (0, I)$ and $\|\cdot\|$ denotes standard Euclidean norm. Then the first $k$ coordinates of $\bs{U}$ satisfy
\begin{align}
\label{LevyRepresentation}
	\sqrt{n} \, (U_1, \ldots , U_k) \
	&\stackrel{d}{=} \ (Z_1, \ldots, Z_k) \Big/
		\biggl( n^{-1} \sum_{j=1}^n Z_j^2 \biggr)^{1/2} \\
\nonumber
	&\rightarrow_d \ (Z_1, \ldots , Z_k) \ \sim \ N_k(0,I_k) ,
\end{align}
since $n^{-1} \sum_{j=1}^n Z_j^2 \rightarrow_p 1$ by the weak law of large numbers. Indeed, let
\[
	Q_{n,k} \ := \ \LL \bigl( r_n (U_1, \ldots , U_k) \bigr)
\]
with $r_n > 0$, and let
\[
	P_k \ := \ \LL(Z_1, \ldots , Z_k) \ = \ N_k(0,I) .
\]
\cite{Diaconis_Freedman_1987} showed that
\[
	\DTV(Q_{n,k}, P_k) \le \frac{k+3}{n-k-3}
	\qquad\text{for} \ \ 1 \le k \le n-4 \ \text{and} \ r_n = \sqrt{n}.
\]
By means of Theorem~\ref{thm:Beta.Gamma}, this bound can be improved by a factor larger than $2$. The approximation becomes even better if we set $r_n = \sqrt{n-2}$. To verify all this, we consider the random variables $R_k := \bigl( \sum_{i=1}^k Z_i^2 \bigr)$, $R_n := \bigl( \sum_{i=1}^n Z_i^2 \bigr)$ and
\[
	\bs{V} \ := \ R_k^{-1} (Z_1,\ldots,Z_k) .
\]
Note that $\bs{V}$ is uniformly distributed on the unit sphere in $\R^k$ and independent of $(R_k, R_n)$. Moreover,
\[
	(Z_1,\ldots,Z_k) \ = \ R_k \bs{V}
	\quad\text{and}\quad
	(U_1,\ldots,U_k) \ = \ \frac{R_k}{R_n} \, \bs{V} .
\]
But $R_k^2 \sim \mathrm{Gamma}(k/2,1/2)$ and $R_k^2/R_n^2 \sim \mathrm{Beta}(k/2,(n-k)/2)$. Hence,
\begin{align*}
	\rho(Q_{n,k},P_k) \
	&= \ \rho \bigl( \LL(r_n R_k/R_n), \LL(R_k) \bigr) \\
	&= \ \rho \bigl( \LL(R_k^2/R_n^2), \LL(r_n^{-2} R_k^2) \bigr) \\
	&= \ \rho \bigl( \mathrm{Beta}(k/2, (n - k)/2), \mathrm{Gamma}(k/2,r_n^2/2) \bigr) .
\end{align*}
Applying Theorem~\ref{thm:Beta.Gamma} with $a := k/2$, $b := (n-k)/2$ and $c := r_n^2/2$ yields the following bounds:

\begin{Corollary}
\label{cor:Levy.Poincare}
For $n > k+2$,
\begin{align*}
	\rho(Q_{n,k},P_k) \
	&< \ (1 - \delta)^{-1/2}
		\qquad\text{and} \\
	\DTV(Q_{n,k},P_k) \
	&< \ 1 - \sqrt{1 - \delta}
		\ < \ \frac{\delta}{2 - \delta} ,
\end{align*}
where
\[
	\delta \ = \ \begin{cases}
		\displaystyle
		\frac{k+2}{n} & \text{if} \ r_n = \sqrt{n} , \\[2ex]
		\displaystyle
		\frac{k}{n-2} & \text{if} \ r_n = \sqrt{n - 2} .
	\end{cases}
\]
\end{Corollary}

Figures~\ref{fig:Levy.Poincare.a} and \ref{fig:Levy.Poincare.b} illustrate Corollary~\ref{cor:Levy.Poincare} in case of $k = 1$. For dimensions $n = 5, 10$, Figure~\ref{fig:Levy.Poincare.a} shows the standard Gaussian density $f$ (green) and the density $g_n$ of $Q_{n,1}$ in case of $r_n = \sqrt{n}$ (black) and $r_n = \sqrt{n-2}$ (blue). Figure~\ref{fig:Levy.Poincare.b} depicts the corresponding ratios $g_n/f$. The dotted black and blue lines are the corresponding upper bounds $(1 - \delta)^{-1/2}$ from Corollary~\ref{cor:Levy.Poincare}. These pictures show clearly that using $r_n = \sqrt{n-2}$ instead of $r_n = \sqrt{n}$ yields a substantial improvement.

\begin{figure}
\includegraphics[width=0.49\textwidth]{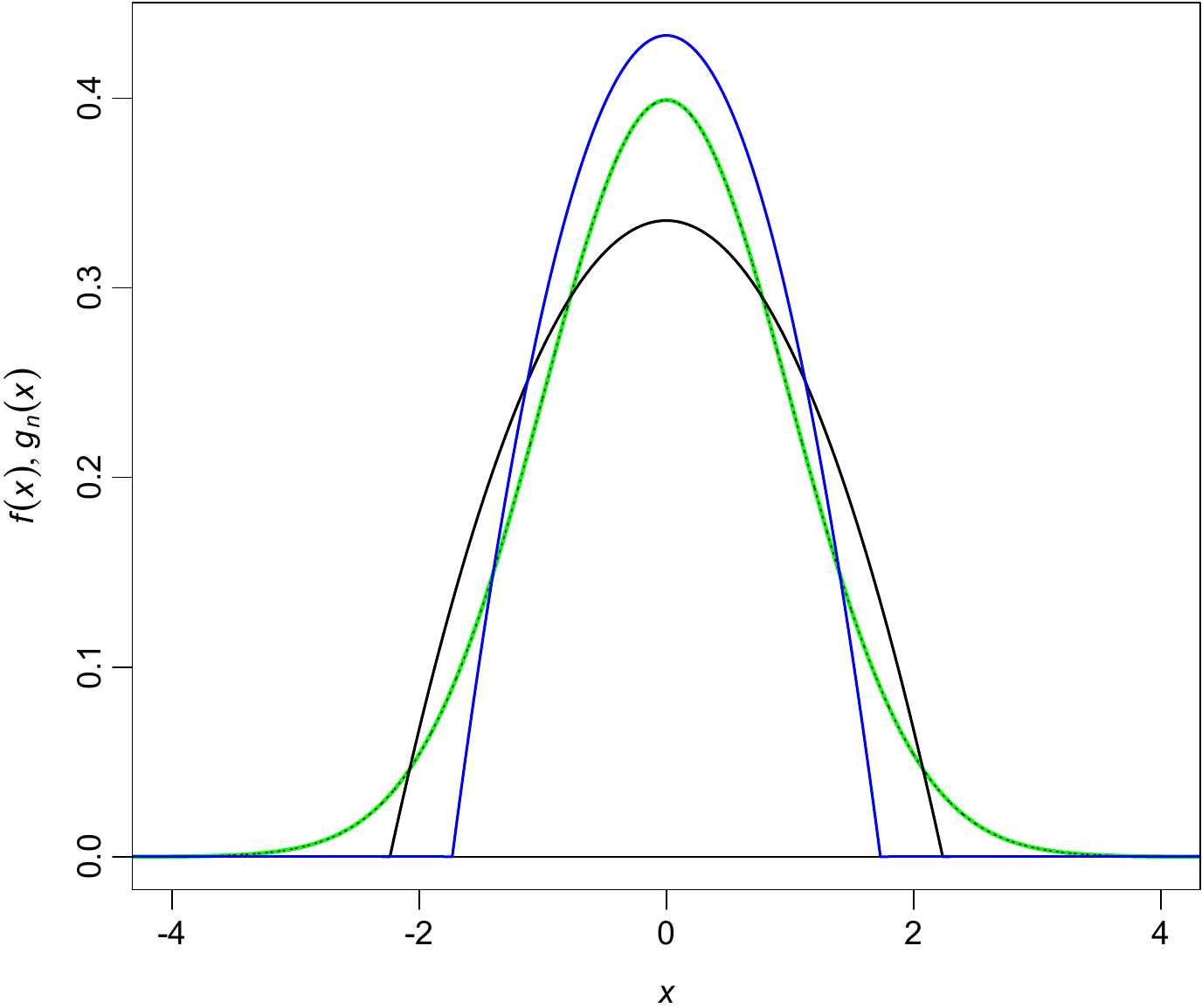}
\hfill
\includegraphics[width=0.49\textwidth]{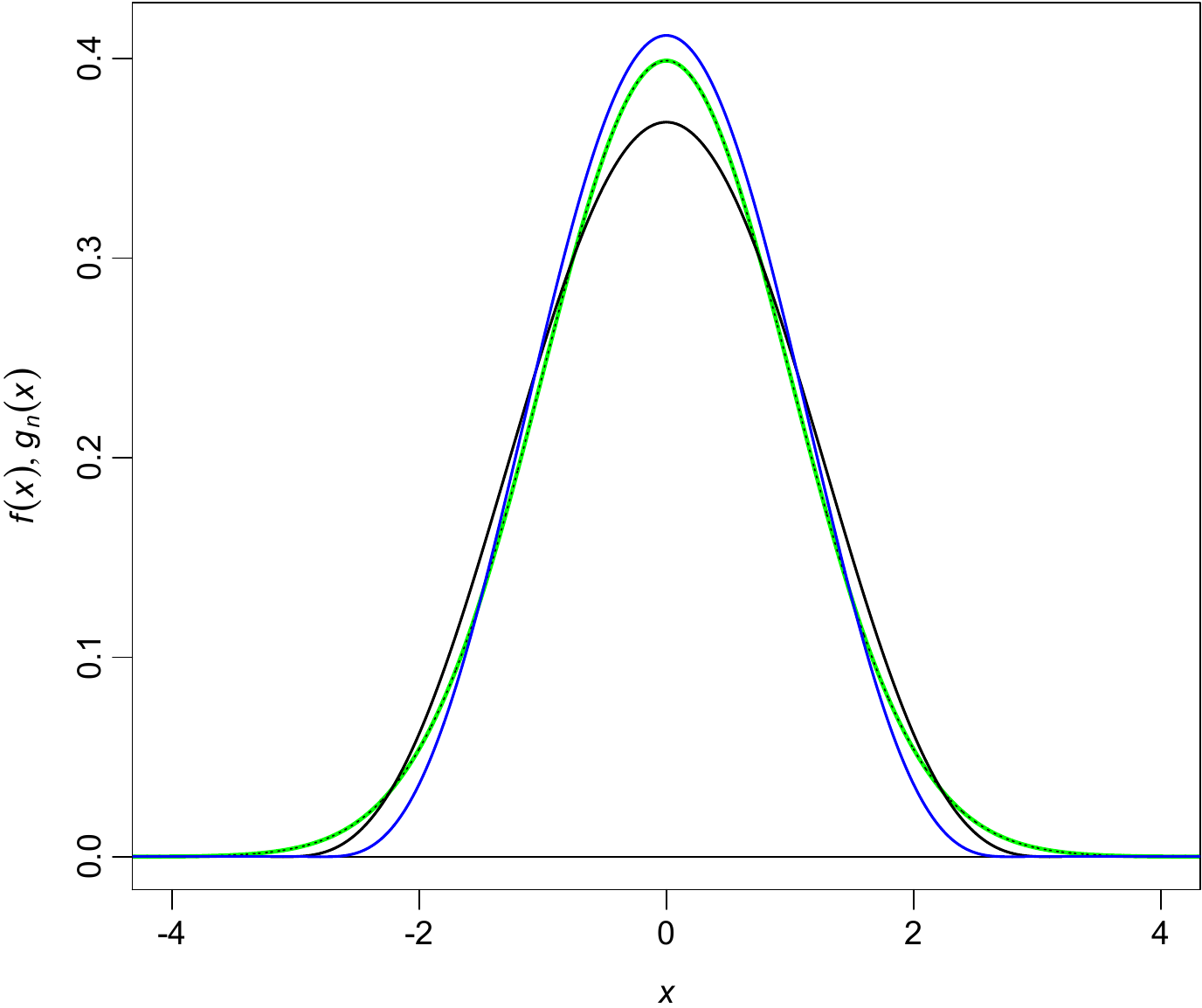}

\caption{Densities of $N(0,1)$ and $Q_{n,1}$ for $n = 5, 10$.}
\label{fig:Levy.Poincare.a}
\end{figure}

\begin{figure}
\includegraphics[width=0.49\textwidth]{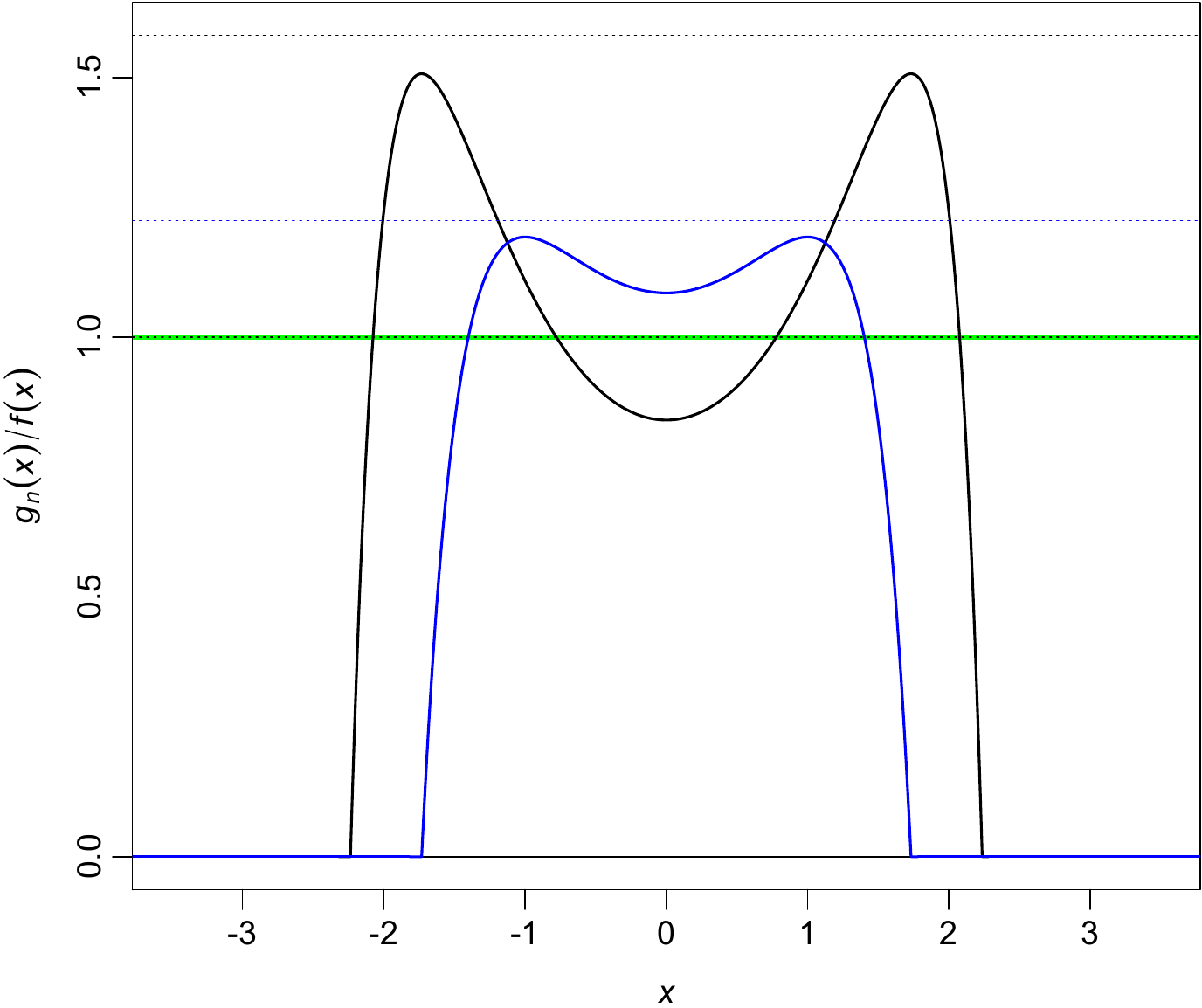}
\hfill
\includegraphics[width=0.49\textwidth]{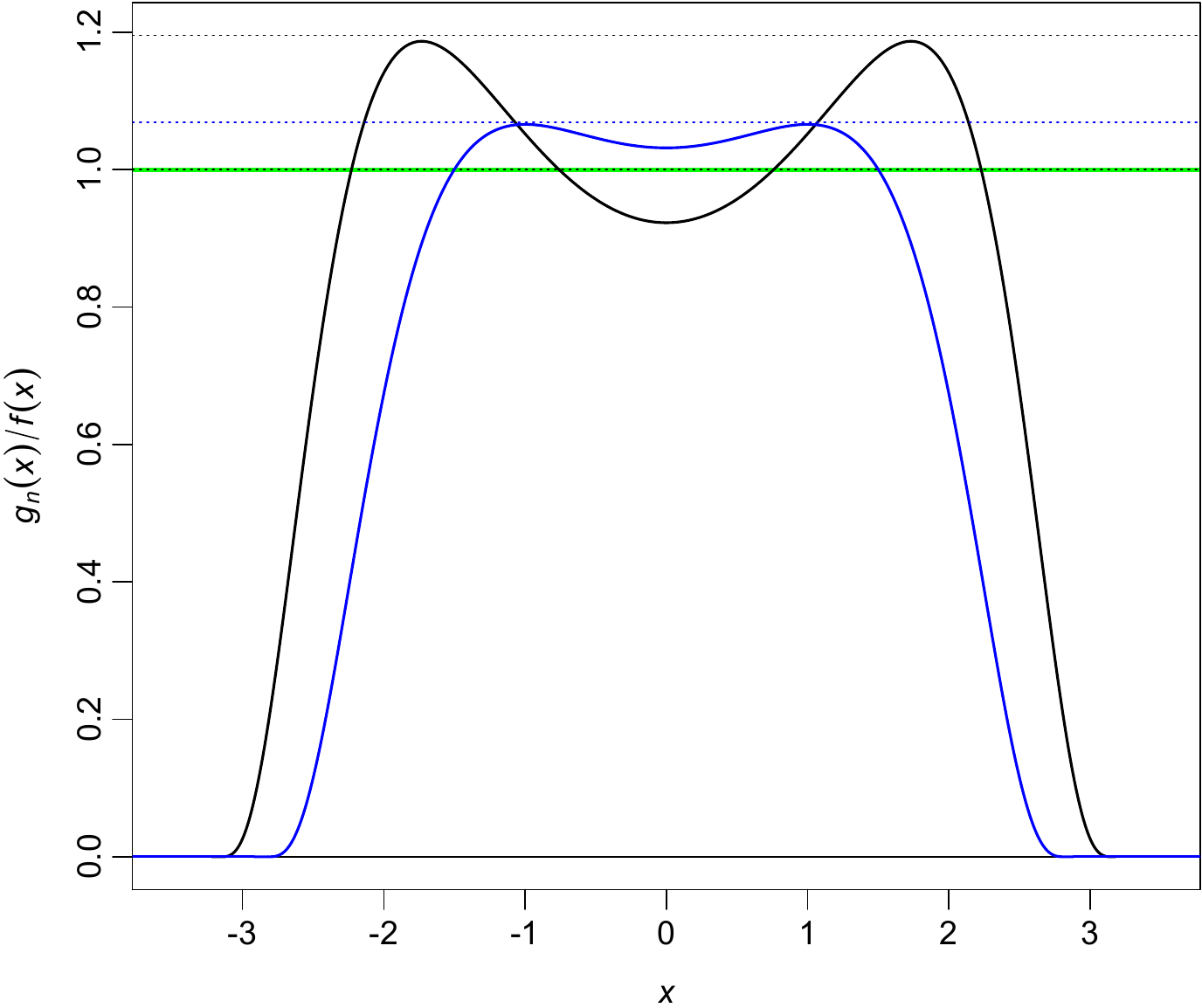}

\caption{Density ratios for Figure~\ref{fig:Levy.Poincare.a}.}
\label{fig:Levy.Poincare.b}
\end{figure}

\subsection{Dirichlet distributions and uniform spacings}

\paragraph{Dirichlet distributions.}
For integers $1 \le k \le N$ and parameters $a_1, \ldots, a_N, c > 0$, let $\bs{X}$ be a random vector with independent components $X_i \sim \mathrm{Gamma}(a_i,c)$. With $X_+ := \sum_{i=1}^N X_i$, it is well-known that the random vector
\[
	\bs{Y} = (Y_1,\ldots,Y_N) \ := \ \Bigl( \frac{X_1}{X_+}, \ldots, \frac{X_N}{X_+} \Bigr)
\]
and $X_+$ are independent, where $X_+ \sim \mathrm{Gamma}(a_+,c)$ with
\[
	a_+ \ := \ \sum_{i=1}^N a_i .
\]
The distribution of $\bs{Y}$ is the Dirichlet distribution with parameters $a_1,\ldots,a_N$, written
\[
	\bs{Y} \ \sim \ \mathrm{Dirichlet}(a_1, \ldots, a_N) .
\]
Now let us focus on the first $k$ components of $\bs{X}$ and $\bs{Y}$:
\begin{align*}
	(X_1,\ldots,X_k) \
	&= \ X_+^{(k)} (V_1,\ldots,V_k) , \\
	(Y_1,\ldots,Y_k) \
	&= \ \frac{X_+^{(k)}}{X_+} (V_1,\ldots,V_k) ,
\end{align*}
with
\[
	X_+^{(k)} \ := \ \sum_{i=1}^k X_i
	\quad\text{and}\quad
	V_i \ := \ \frac{X_i}{X_+^{(k)}} .
\]
Then $(V_1,\ldots,V_k) \sim \mathrm{Dirichlet}(a_1, \ldots, a_k)$ and is independent of $(X_+^{(k)}, X_+)$, while
\[
	\frac{X_+^{(k)}}{X_+} \ \sim \ \mathrm{Beta}(a_+^{(k)},a_+ - a_+^{(k)})
	\quad\text{and}\quad
	X_+^{(k)} \ \sim \ \mathrm{Gamma}(a_+^{(k)}, c)
\]
with
\[
	a_+^{(k)} \ := \ \sum_{i=1}^k a_i .
\]
Hence, the difference between $\LL(Y_1,\ldots,Y_k)$ and $\LL(X_1,\ldots,X_k)$, in terms of the ratio measure, is the difference between $\mathrm{Beta}(a_+^{(k)},a_+ - a_+^{(k)})$ and $\mathrm{Gamma}(a_+^{(k)}, c) $. Thus Theorem~\ref{thm:Beta.Gamma} yields the following bounds:

\begin{Corollary}
\label{cor:Dirichlet}
Let $P_k := \otimes_{i=1}^k \mathrm{Gamma}(a_i,c)$, and let $Q_{N,k} := \LL(Y_1,\ldots,Y_k)$. Then
\begin{align*}
	\rho(Q_{N,k},P_k) \
	&< \ (1 - \delta)^{-1/2}
		\qquad\text{and} \\
	\DTV(Q_{N,k},P_k) \
	&< \ 1 - \sqrt{1 - \delta} \ < \ \frac{\delta}{2 - \delta} ,
\end{align*}
where either
\[
	c \ = \ a_+
	\quad\text{and}\quad
	\delta \ = \ \frac{a_+^{(k)} + 1}{a_+} ,
\]
or
\[
	c \ = \ a_+ - 1
	\quad\text{and}\quad
	\delta \ = \ \frac{a_+^{(k)}}{a_+ - 1} .
\]
\end{Corollary}

\paragraph{Uniform spacings.}
A special case of the previous result are uniform spacings: For an integer $n \ge 2$, let $U_1, \ldots, U_n$ be independent random variables with uniform distribution on $[0,1]$. Then we consider the order statistics $0 < U_{n:1} < U_{n:2} < \cdots < U_{n:n} < 1$. 
With $U_{n:0} := 0$ and $U_{n:n+1} := 1$, it is well-known that
\[
	(U_{n:j} - U_{n:j-1})_{j=1}^{n+1}
	\ \sim \ \mathrm{Dirichlet}(\underbrace{1,1,\ldots,1}_{n+1 \ \text{times}}) .
\]
That means, the $n+1$ spacings have the same distribution as $(E_j/E_+)_{j=1}^{n+1}$ with independent, standard exponential random variables $E_1, \ldots, E_{n+1}$ and $E_+ := \sum_{j=1}^{n+1} E_j$. Consequently, Corollary~\ref{cor:Dirichlet} and the second remark after Theorem~\ref{thm:Beta.Gamma} yield the following bounds:

\begin{Corollary}
\label{cor:Unif.spacings}
For integers $1 \le k < n$ let $Q_{n,k}$ be the distribution of the vector
\[
	Y_{n,k} \ := \ n (U_{n:j} - U_{n:j-1})_{j=1}^k .	
\]
Further let $P_k$ be the $k$-fold product of the standard exponential distribution. Then
\[
	\rho(Q_{n,k}, P_k)
	\ \le \ \begin{cases}
		\displaystyle
		\exp \Bigl( \frac{1}{2n} + \frac{1}{4n^2} \Bigr)
			& \text{if} \ k = 1 , \\[2ex]
		\displaystyle
		\Bigl( 1 - \frac{k}{n} \Bigr)^{-1/2}
			& \text{in general}.
		\end{cases}
\]
In particular,
\[
	\DTV(Q_{n,k},P_k)
	\ \le \ \begin{cases}
		\displaystyle
		\frac{1}{2n} + \frac{1}{4n^2}
			& \text{if} \ k = 1 , \\[2ex]
		\displaystyle
		1 - \sqrt{1 - \frac{k}{n}}
		\ < \ \frac{k}{2n - k}
			& \text{in general} .
	\end{cases}
\]
\end{Corollary}

\paragraph{Remarks.}
Corollary~\ref{cor:Unif.spacings} gives another proof of the results of \nocite{Runnenburg_Vervaat_1969}{Runnenburg and Vervaat (1969)},
who obtained bounds on $\DTV (Q_{n,k}, P_k)$ by first bounding the Kullback--Leibler divergence;  see their Remark 4.1, pages 74--75. It can be shown via the methods of \nocite{Hall_Wellner_1979}{Hall and Wellner (1979)} that 
\[
	\DTV(Q_{n,1},P_1) \ \le \ \frac{2 e^{-2}}{n} + \frac{e^{-2}}{n^2} ,
\]
where $2 e^{-2} \approx .2707 < 1/2$.

\subsection{Student distributions}
\label{subsec:StudentDistributions}

For $r > 0$ let $t_r$ denote student's t distribution with $r$ degrees of freedom, with density
\[
	f_r(x) \ = \ \frac{\Gamma((r+1)/2)}{\Gamma(r/2) \sqrt{r\pi}}
		\Bigl( 1 + \frac{x^2}{r} \Bigr)^{-(r+1)/2} .
\]
It is well-known that $f_r$ converges uniformly to the density $\phi$ of the standard Gaussian distribution $N(0,1)$, where $\phi(x) := \exp(- x^2/2) / \sqrt{2\pi}$. The distribution $t_r$ has heavier tails than the standard Gaussian distribution and, indeed,
\[
	\rho \bigl( t_r, N(0,1) \bigr)
	\ = \ \infty .
\]
However, for the reverse ratio measure we do obtain a reasonable upper bound:

\begin{Lemma}
\label{lem:N01.student}
For $r \geq 2$,
\[
	\frac{1}{2(r+1)} \ < \ \log \rho(N(0,1), t_r) \ < \ \frac{1}{2r} .
\]
\end{Lemma}

\paragraph{Remarks.}
It follows from Lemma~\ref{lem:N01.student} that
\[
	r \log \rho(N(0,1), t_r) \ \to \ \frac{1}{2}
	\quad\text{as} \ r \to \infty .
\]
By means of Proposition~\ref{prop:rho.and.divergences}~(a) we obtain the inequality 
$r \DTV(N(0,1),t_r) \le 1/2$ for $r \ge 2$. \cite{Pinelis_2015} proved that
\[
	r \DTV(N(0,1),t_r)
	\ < \ C := \frac{1}{2} \sqrt{ \frac{7 + 5\sqrt{2}}{\pi e^{1+\sqrt{2}}} } \approx 0.3165
\]
for $r\ge 4$, and that $r \DTV \bigl ( N(0,1), t_r \bigr) \rightarrow C$ as 
$r \rightarrow \infty$. So $C$ is optimal in the bound for $\DTV$, 
whereas $1/2$ is optimal for $\rho$.

Let $Z$ and $T_r$ be random variables with distribution $N(0,1)$ and 
$t_r$, respectively, where $r \ge 2$. Then for any Borel set $B \subset \R$,
\[
	\Pr(T_r \in B) \ \ge \ e^{-1/(2r)} P(Z \in B) .
\]
In particular,
\[
	\left.\begin{array}{c}
		\Pr \bigl( \pm T_r < \Phi^{-1}(1 - \alpha) \bigr) \\[1ex]
		\Pr \bigl( |T_r| < \Phi^{-1}(1 - \alpha/2) \bigr)
	\end{array}\!\!\right\} \ \ge \ e^{-1/(2r)} (1 - \alpha) .
\]

\subsection{A counterexample: convergence of normal extremes}

In all previous settings, we derived upper bounds for $\rho(Q,P)$ which implied resonable bounds for $\DTV(Q,P) = \DTV(P,Q)$, whereas $\rho(P,Q) = \infty$ in general. This raises the question whether there are probability densities $g$ and $f_n$, $n \ge 1$, such that $\DTV(f_n,g) \to 0$, but both $\rho(f_n, g) = \infty$ and $\rho (g, f_n) = \infty$? The answer is ``yes'' in view of the following example.

\begin{Example}
Suppose that $Z_1, Z_2, Z_3, \ldots$ are independent, standard Gaussian random variables. Let $V_n := \max\{Z_i : 1 \le i \le n\}$.  Let $b_n>0$ satisfy $2\pi b_n^2 \exp(b_n^2) = n^2$ and then set $a_n := 1/b_n$. Then it is well-known that
\begin{eqnarray} 
	Y_n := (V_n -b_n)/a_n \ \rightarrow_d \ Y_{\infty} \sim G 
	\label{MaxOfGaussConvergeToGumbel}
\end{eqnarray}
where $G$ is the Gumbel distribution function given by $G(x) = \exp(-\exp(-x))$. Set $F_n(x) := P(Y_n \le x) $ for $n\ge 1$ and $x \in \R$.
\cite{Hall_1979} shows that for constants $0 < C_1 < C_2 \le 3$ and sufficiently large $n$,
\[ 
	\frac{C_1}{\log n}
	\ < \ \| F_n - G \|_{\infty} := \sup_{x \in \R} |F_n (x) - G(x)|
	\ < \ \frac{C_2}{\log n} ,
\]
and $d_{\mathrm{L}}(F_n , G) = O(1/\log n)$ for the L\'evy metric $d_{\mathrm{L}}$. It is also known that if $\tilde{b}_n := (2 \log n )^{1/2} - (1/2) \{ \log \log n + \log (4 \pi) \} /(2 \log n)^{1/2}$ and $\tilde{a}_n := 1/\tilde{b}_n$, then $\tilde{a}_n / a_n \rightarrow 1$,  $(\tilde{b}_n - b_n )/a_n \rightarrow 0$ and \eqref{MaxOfGaussConvergeToGumbel} continues to hold with $a_n$ and $b_n$ replaced by $\tilde{a}_n $ and $\tilde{b}_n$, but the rate of convergence in the last display is not better than $(\log \log n)^2/\log n$.
 
In this example the densities $f_n$ of $F_n$ are given by 
\begin{align*}
	f_n(x) \ = \ \Phi (a_n x + b_n )^n
		\frac{n a_n \phi (a_n x + b_n)}{\Phi (a_n x + b_n)} 
	\rightarrow \ G(x) \cdot e^{-x} = G'(x) =: g(x)
\end{align*}
for each fixed $x \in \R$; here $\phi$ is the standard normal density and $\Phi(z) :=\int_{-\infty}^z \phi (y) dy$ is the standard normal distribution function. Thus $\DTV(F_n, G) \rightarrow 0$ by Scheff\'e's lemma. But in this case it is easily seen that both $\rho(f_n, g) = \infty$ and $\rho(g, f_n) = \infty$ where the infinity in the first case occurs in the left tail,  and the infinity in the second case occurs in the right tail.  

We do not know a rate for the total variation convergence in this example, but it cannot be faster than $1/\log n$.
\end{Example}

\section{Proofs and Auxiliary Results}
\label{sec:Proofs}

\subsection{Proofs of the main results}

\begin{proof}[\bf Proof of \eqref{eq:rho.via.LR}]
Suppose that $\mu(\{g/f > r\}) = 0$ for some real number $r > 0$. Then $g \le r f$, $\mu$-almost everywhere, so $Q(A) \le r P(A)$ for all $A \in \AA$, and this implies that $\rho(Q,P) \le r$. On the other hand, if $\mu(\{g/f \ge r\}) > 0$ for some real number $r > 0$, then $A := \{g/f \ge r\} = \{g \ge r f\} \cap \{g > 0\}$ satisfies $Q(A) > 0$ and $Q(A) \ge r P(A)$, whence $\rho(Q,P) \ge r$. These considerations show that $\rho(Q,P)$ equals the $\mu$-essential supremum of $g/f$.
\end{proof}

\begin{proof}[\bf Proof of Proposition~\ref{prop:rho.and.divergences}]
\textbf{(a)} Under the given hypotheses that $\psi$ is non-decreasing with $\psi(1) = 0$ and $g/f \leq \rho$, we have
\begin{equation}
\label{eq:rhoincreasing}
	\int \psi(g/f) \, dQ
	\ \le \ \int_{\{g > f\}} \psi(g/f) \, dQ
	\ \le \ Q(\{g > f\})\psi(\rho) .
\end{equation}
Equality holds in the first inequality if and only if $Q\bigl(\{g < f\} \cap \{\psi(g/f) < 0\}\bigr) = 0$, and in the second inequality if and only if $Q\bigl(\{g > f\} \cap \{\psi(g/f) < \psi(\rho)\}\bigr) = 0$. In particular, if $g/f \in \{0,\rho\}$, then $Q\bigl(\{g < f\}) = Q(\{g/f = 0\}) = 0$ and $Q\bigl(\{g > f\} \cap \{\psi(g/f) < \psi(\rho)\}\bigr) = Q(\emptyset) = 0$, so we have equality in~\eqref{eq:rhoincreasing}.

\noindent
\textbf{(b)} For any convex function $\psi:[0,\infty) \rightarrow \mathbb{R}$ and $y \in [0,\rho]$, we have
\[
  \psi(y) \ \le \ \psi(0) + \frac{y}{\rho}\{\psi(\rho) - \psi(0)\}
\]
with equality in case of $y \in \{0,\rho\}$. Hence
\[
	\int \psi(g/f) \, dP
	\ \le \ \psi(0) + \frac{\psi(\rho) - \psi(0)}{\rho} \int \frac{g}{f} \, dP
	\ = \ \psi(0) + \frac{\psi(\rho) - \psi(0)}{\rho} .
\]
Equality holds if $g/f \in \{0,\rho\}$.
\end{proof}

\begin{proof}[\bf Proof of \eqref{ineq:rho.sampling} and comparison with \eqref{ineq:Freedman.sampling}]
The asserted bounds are trivial in case of $n = 1$, so we assume that $2 \le n \le N$. Note first that
\[
	\log \rho(Q,P)
	\ = \ \log( N^n/[N]_n)
	\ = \ \sum_{j=1}^{n-1} H(j)
\]
with $H(x) := - \log(1 - x/N) = \sum_{\ell=1}^\infty (x/N)^\ell/\ell$ for $x \ge 0$. Since $H(x) \ge x/N$,
\[
	\log \rho(Q,P)
	\ \ge \ \sum_{j=1}^{n-1} j/N \ = \ \frac{n(n-1)}{2N} .
\]
This is essentially \nocite{Freedman_1977}{Freedman's (1977)} argument. For the upper bound, it suffices to show that for $1 \le n < N$, the increment
\begin{equation}
\label{eq:increment.1}
	\log(N^{n+1} / [N]_{n+1}) - \log(N^n/[N]_n)
	\ = \ H(n)
\end{equation}
is not larger than the increment
\begin{equation}
\label{eq:increment.2}
	- \frac{n+1}{2} \log \Bigl( 1 - \frac{n}{N} \Bigr)
	+ \frac{n}{2} \log \Bigl( 1 - \frac{n-1}{N} \Bigr)
	\ = \ (n+1) H(n)/2 - n H(n-1)/2 .
\end{equation}
But the difference between \eqref{eq:increment.2} and \eqref{eq:increment.1} equals
\[
	(n-1) H(n)/2 - n H(n-1)/2
	\ = \ n(n-1) \bigl( H(n)/n - H(n-1)/(n-1) \bigr) / 2
	\ \ge \ 0 ,
\]
because $H(x)/x$ is non-decreasing on $[0,\infty)$. Since $H(tx) > t H(x)$ for $x \in [0,N)$ and $t > 1$, we may also conclude that for $3 \le n \le N$,
\[
	- \log \Bigl( 1 - \frac{n(n-1)}{2N} \Bigr) \ = \ H(n(n-1)/2)
	\ > \ (n/2) H(n-1) \ = \ - \frac{n}{2} \log \Bigl( 1 - \frac{n-1}{N} \Bigr) .
\]
\end{proof}

\paragraph{Auxiliary inequalities.}
In what follows, we will use repeatedly the following inequalities for logarithms: For real numbers $x, a > 0$ and $b > - x$,
\begin{align}
\label{ineq:log.ratio.1}
	(x + b) \log \Bigl( \frac{x}{x + a} \Bigr) \
	&< \ - a + \frac{a(a-2b)}{2x + a} - \frac{2a^3 (x+b)}{3 (2x + a)^3} \\
\label{ineq:log.ratio.2}
	&< \ - a + \frac{a(a-2b)}{2x + a}
\intertext{and}
\label{ineq:log.ratio.3}
	(x + a/2) \log \Bigl( \frac{x}{x + a} \Bigr) \
	&> \ - a - \frac{a^3}{12 x(x + a)} . 
\end{align}
These inequalities follow essentially from the fact
\[
	\log \Bigl( \frac{x}{x + a} \Bigr)
	\ = \ \log \Bigl( \frac{2x + a - a}{2x + a + a} \Bigr)
	\ = \ \log \Bigl( \frac{1 - y}{1 + y} \Bigr)
	\ = \ - 2 \sum_{\ell=0}^\infty \frac{y^{2\ell+1}}{2\ell + 1}
	\ < \ -2y - \frac{2y^3}{3}
\]
with $y := a/(2x + a)$, where the Taylor series expansion in the second to last step is well-known and follows from the usual expansion $\log(1 \pm y) = - \sum_{k=1}^\infty (\mp y)^k / k$. Then it follows from $x + b > 0$ that
\[
	(x + b) \log \Bigl( \frac{x}{x + a} \Bigr)
	\ < \ - \frac{2a(x + b)}{2x + a} - \frac{2a^3(x + b)}{3 (2x + a)^3}
	\ = \ - a + \frac{a(a - 2b)}{2x + a} - \frac{2a^3(x + b)}{3 (2x + a)^3} ,
\]
whereas
\begin{align*}
	(x + a/2) \log \Bigl( \frac{x}{x+a} \Bigr) \
	&= \ \frac{a}{2y} \log \Bigl( \frac{1 - y}{1 + y} \Bigr)
		\ = \ - a \sum_{\ell=0}^\infty \frac{y^{2\ell}}{2\ell + 1} \\
	&> \ - a - \frac{ay^2}{3(1 - y^2)}
		\ = \ - a - \frac{a^3}{12x (x + a)} .
\end{align*}

Here is another expression which will be encountered several times: For $\delta \in [0,1]$,
\[
	1 - \sqrt{1 - \delta} \ = \ \frac{\delta}{1 + \sqrt{1 - \delta}}
	\ = \ \frac{\delta}{2 - (1 - \sqrt{1 - \delta})}
	\ = \ \cdots \ = \ \frac{\delta}{2 - \frac{\delta}{2 - \frac{\delta}{2 - \cdots}}} ,
\]
and the inequality $\sqrt{1 - \delta} \ge 1 - \delta$ implies that
\begin{equation}
\label{eq:sqrt.delta}
	1 - \sqrt{1 - \delta} \ \le \ \frac{\delta}{2 - \delta}
	\ = \ \frac{\delta}{2} \Bigl( 1 - \frac{\delta}{2} \Bigr)^{-1}
	\ = \ \frac{\delta}{2} + \frac{\delta^2}{4 - 2\delta} .
\end{equation}

Recall that we write $[a]_0 := 1$ and $[a]_m := \prod_{i=0}^{m-1} (a - i)$ for real numbers $a$ and integers $m \ge 1$. In particular, $\binom{n}{k} = [n]_k / k!$ for integers $0 \le k \le n$.

\begin{proof}[\bf Proof of Theorem~\ref{thm:HypBin}]
The assertions are trivial in case of $n = 1$ or $L \in \{0,N\}$, because then $\Hyp(N,L,n) = \Bin(n,L/N)$. Hence it suffices to consider $n \ge 2$ and $1 \le L \le N-1$. For $k \in \{0,1,\ldots,n\}$ let 
\begin{align*}
	h(k) = h_{N,L,k}(k) \
	:= \ &\Hyp(N,L,n)(\{k\})
		\ = \ \binom{L}{k} \binom{N-L}{n-k} \Big/ \binom{N}{n} \\
	 = \ &\binom{n}{k} \frac{[L]_k [N-L]_{n-k}}{[N]_n} , \\
	b(k) = b_{n,L/N}(k) \
	:= \ &\Bin(n,L/N)(\{k\})
		\ = \ \binom{n}{k} (L/N)^k (1 - L/N)^{n-k} \\
	 = \ &\binom{n}{k} \frac{L^k (N-L)^{n-k}}{N^n}
\intertext{and}
	r(k) = r_{N,L,n}(k) \
	:= \ &\frac{h(k)}{b(k)}
		\ = \ \frac{[L]_k [N - L]_{n-k} N^n}{L^k (N-L)^{n-k} [N]_n} .
\end{align*}
Since
\[
	r_{N,N-L,n}(n-k) \ = \ r_{N,L,n}(k) ,
\]
it even suffices to consider
\[
	n \ge 2
	\quad\text{and}\quad
	1 \le L \le N/2 .
\]
In this case, $r(k) > 0$ for $1 \le k \le \min(n,L)$, and $r(k) = 0$ for $\min(n,L) < k \leq n$.

In order to maximize the weight ratio $r$, note that for any integer $0 \le k < \min(L,n)$,
\[
	\frac{r(k+1)}{r(k)} \ = \ \frac{(L - k)(N - L)}{L(N - L - n + k + 1)}
	\ \left\{\!\!\begin{array}{c} \le \\ > \end{array} \!\! \right\} \ 1
\]
if and only if
\[
	k \ \left\{\!\!\begin{array}{c} \ge \\ < \end{array}\!\!\right\} \ \frac{(n-1)L}{N} .
\]
Consequently,
\begin{align*}
	\rho \bigl( &\Hyp(N,L,n), \Bin(n,L/N) \bigr) \ = \ r_{N,L,n}(k) \\
	&\text{with}\quad
		k = k_{N,L,n} \ := \ \Bigl\lceil \frac{(n-1)L}{N} \Bigr\rceil
		\ \in \ \{1,\ldots,n-1\} .
\end{align*}

The worst-case value $k_{N,L,n}$ equals $1$ if and only if $L \le N/(n-1)$. But
\begin{align*}
	r_{N,L,n}(1) \
	&= \ \frac{[N - L]_{n-1} N^n}{(N - L)^{n-1} [N]_n} 
		\ = \ \prod_{i=0}^{n-2} \Bigl( 1 - \frac{i}{N - L} \Bigr)
			\frac{N^n}{[N]_{n}} \\
	&\le \ \prod_{i=0}^{n-2} \Bigl( 1 - \frac{i}{N - 1} \Bigr)
			\frac{N^n}{[N]_{n}} 
	    \ = \ (1 - 1/N)^{-(n-1)} \ = \ r_{N,1,n}(1) .
\end{align*}
Consequently, it suffices to consider
\[
	N/(n-1) < L \le N/2 .
\]
Note that these inequalities for $L$ imply that $n - 1 > 2$. Hence it remains to prove the assertions when $n \ge 4$ and $N/(n-1) < L \le N/2$.

The case $n = 4$ is treated separately: Here it suffices to show that
\[
	r_{N,L,4}(2) \ \le \ r_{N,1,4}(1)
	\quad\text{for} \ N \ge 6 \ \text{and} \ 1 < L \le N/2 .
\]
Indeed
\begin{align*}
	\frac{r_{N,L,4}(2)}{r_{N,1,4}(1)} \
	&= \ \frac{[L]_2[N-L]_2 (N-1)^3}{L^2 (N - L)^2 [N-1]_3} 
	      \ = \ \frac{(L-1)(N-L-1) (N - 1)^2}{L(N-L) (N-2) (N-3)} \\
	&= \ \frac{(L(N-L) - N + 1)(N-1)^2}{L(N-L) ((N-1)^2 - 3N + 5)} 
	      \ = \ \Bigl( 1 - \frac{N-1}{L(N-L)} \Bigr)
		    \Big/ \Bigl( 1 - \frac{3N - 5}{(N-1)^2} \Bigr) \\
	&\le \ \Bigl( 1 - \frac{4(N-1)}{N^2} \Bigr)
		\Big/ \Bigl( 1 - \frac{3N - 5}{(N-1)^2} \Bigr)
\end{align*}
with equality if and only if $L = N/2$. The latter expression is less than or equal to $1$ if and only if
\[
	\frac{4(N-1)}{N^2} \ \ge \ \frac{3N - 5}{(N-1)^2} ,
\]
and elementary manipulations show that this is equivalent to
\[
	(N - 7/2)^2 + 12 - 49/4 \ \ge \ 4/N .
\]
But this inequality is satisfied for all $N \ge 5$.

Consequently, it suffices to prove our assertion in case of
\[
	n \ge 5
	\quad\text{and}\quad
	N/(n-1) < L \le N/2 .
\]
The maximizer $k = k_{N,L,n}$ of the density ratio is $k = \lceil (n-1)L/N \rceil \ge 2$, and
\[
	n - k
	\ = \ \lfloor n - (n - 1)L/N \rfloor
	\ \ge \ \lfloor n - (n - 1)/2 \rfloor
	\ = \ \lfloor (n+1)/2 \rfloor
	\ \ge \ 3 .
\]
Now our task is to bound
\begin{align*}
	\log \rho \bigl( & \Hyp(N,L,n), \Bin(n,L/N) \bigr) \\
	&= \ \log \Bigl( \frac{[L]_k}{L^k} \Bigr)
		+ \log \Bigl( \frac{[N - L]_{n-k}}{(N - L)^{n-k}} \Bigr)
		- \log \Bigl( \frac{[N]_n}{N^n} \Bigr) \\
	&= \ \log \Bigl( \frac{[L-1]_{k-1}}{L^{k-1}} \Bigr)
		+ \log \Bigl( \frac{[N - L-1]_{n-k-1}}{(N - L)^{n-k-1}} \Bigr)
		- \log \Bigl( \frac{[N-1]_{n-1}}{N^{n-1}} \Bigr)
\end{align*}
from above. Corollary~\ref{cor:log.Gamma.increments} in Section~\ref{subsec:Log.Gamma} implies that for integers $A \ge m \ge 2$,
\begin{align*}
	\log \Bigl( \frac{[A-1]_{m-1}}{A^{m-1}} \Bigr) \
	= \ &\log((A-1)!) - \log((A - m)!) - (m - 1) \log(A) \\
	= \ &(A - 1/2) \log(A) - A - (m - 1) \log(A) \\
		&- \ (A - m + 1/2) \log(A - m + 1) + A - m + 1 + s_{m,A}^{} \\
	= \ &(A - m + 1/2) \log \Bigl( \frac{A}{A - m + 1} \Bigr) + 1 - m + s_{m,A}^{} ,
\end{align*}
where
\[
	- \frac{m-1}{12A(A-m+1)} \ < \ s_{m,A} \ < \ 0 .
\]
Consequently,
\begin{align*}
	\log \rho \bigl( & \Hyp(N,L,n), \Bin(n,L/N) \bigr) \\
	< \ &(L - k + 1/2) \log \Bigl( \frac{L}{L-k+1} \Bigr)
		+ (N - L - n + k + 1/2) \log \Bigl( \frac{N - L}{N - L - n + k + 1} \Bigr) \\
		&+ \ 1 - (N - n + 1/2) \log \Bigl( \frac{N}{N-n+1} \Bigr)
			+ \frac{n-1}{12 N(N - n + 1)} .
\end{align*}
Now we introduce the auxiliary quantities
\[
	\delta \ := \ \frac{n-1}{N} , \quad \Delta \ := \ 1 - \delta \ = \ \frac{N-n+1}{N}  
\]
and write
\[
	k \ = \ (n-1)L/N + \gamma \ = \ L\delta + \gamma
	\quad\text{with} \ 0 \le \gamma < 1 .
\]
Then
\[
	L - k \ = \ L\Delta - \gamma , \quad
	N - L - n + k \ = \ (N - L) \Delta + \gamma - 1 ,
\]
whence
\begin{align*}
	&(L - k + 1/2) \log \Bigl( \frac{L}{L-k+1} \Bigr)
		+ (N - L - n + k + 1/2) \log \Bigl( \frac{N - L}{N - L - n + k + 1} \Bigr) \\
	&= \ (L\Delta + 1/2 - \gamma) \log \Bigl( \frac{L}{L\Delta + 1 - \gamma} \Bigr)
		 + \bigl((N - L)\Delta + \gamma - 1/2\bigr)
			\log \Bigl( \frac{N - L}{(N - L)\Delta + \gamma} \Bigr) \\
	&= \ (L\Delta + 1/2 - \gamma) \log \Bigl( \frac{L\Delta}{L\Delta + 1 - \gamma} \Bigr)
		 + \bigl((N - L)\Delta + \gamma - 1/2\bigr)
			\log \Bigl( \frac{(N - L)\Delta}{(N - L)\Delta + \gamma} \Bigr) \\
	&\qquad - \ (N - n + 1) \log(\Delta) .		
\end{align*}
It follows from \eqref{ineq:log.ratio.2} with $x = L\Delta$, $a = 1 - \gamma$ and $b = 1/2 - \gamma$ that
\[
	(L\Delta + 1/2 - \gamma) \log \Bigl( \frac{L\Delta}{L\Delta + 1 - \gamma} \Bigr)
	\ < \ - (1 - \gamma) + \frac{\gamma(1 - \gamma)}{2L\Delta + 1 - \gamma} ,
\]
and with $x = (N - L)\Delta$, $a = \gamma$ and $b = \gamma - 1/2$ we may conclude that
\[
	\bigl((N - L)\Delta + \gamma - 1/2\bigr)
		\log \Bigl( \frac{(N - L)\Delta}{(N - L)\Delta + \gamma} \Bigr)
	\ < \ - \gamma + \frac{\gamma(1 - \gamma)}{2(N-L)\Delta + \gamma} .
\]
Hence
\begin{align*}
	\log \rho \bigl( & \Hyp(N,L,n), \Bin(n,L/N) \bigr) \\
	< \ &- \, (1 - \gamma) + \frac{\gamma(1 - \gamma)}{2L\Delta + 1 - \gamma}
		- \gamma + \frac{\gamma(1 - \gamma)}{2(N-L)\Delta + \gamma}
		- (N - n + 1) \log(\Delta) \\
		&+ \ 1 - (N - n + 1/2) \log \Bigl( \frac{N}{N-n+1} \Bigr)
			+ \frac{n-1}{12 N(N - n + 1)} \\
	= \ & g(L) - \frac{\log(\Delta)}{2} + \frac{\delta}{12 N \Delta} ,
\end{align*}
where
\begin{align*}
	g(L) \
	:= \ & \gamma(1 - \gamma) \Bigl( \frac{1}{2L\Delta + 1 - \gamma}
		+ \frac{1}{2(N-L)\Delta + \gamma} \Bigr) \\
	< \ &\frac{1}{8L\Delta} + \frac{1}{8(N-L)\Delta} \ = \ \frac{N}{8L(N-L)\Delta} ,
\end{align*}
because $\gamma(1 - \gamma) \le 1/4$. It will be shown later that
\begin{equation}
\label{eq:Ooff}
	g(L) \ \le \ \frac{\delta}{7\Delta} .
\end{equation}
Consequently,
\begin{align*}
	\log \rho \bigl( \Hyp(N,L,n), \Bin(n,L/N) \bigr) \
	&< \ - \frac{\log(\Delta)}{2} + \frac{\delta}{7\Delta} + \frac{\delta}{12 N\Delta} \\
	&= \ - \frac{\log(1 - \delta)}{2} + \frac{\delta}{7(1 - \delta)}
		+ \frac{\delta}{12 N(1 - \delta)} \\
	&\le \ - \frac{\log(1 - \delta)}{2} + \frac{\delta}{7(1 - \delta)}
		+ \frac{\delta}{6 N} ,
\end{align*}
because $\delta \le 1/2$, and we want to show that the right-hand side is not greater than
\[
	- (n - 1) \log(1 - 1/N) \ = \ (n-1) \sum_{\ell=1}^\infty \frac{1}{\ell N^\ell}
	\ > \ \delta + \frac{\delta}{2N} .
\]
Hence, it suffices to show that
\[
	- \frac{\log(1 - \delta)}{2} + \frac{\delta}{7(1 - \delta)} - \delta \ \le \ 0 .
\]
But the left-hand side is a convex function of $\delta \in [0,1/2]$ and takes the value $0$ for $\delta = 0$. Thus it suffices to verify that the latter inequality holds for $\delta = 1/2$. Indeed, for $\delta = 1/2$, the left-hand side is $\log(2)/2 + 1/7 - 1/2 = (\log(2) - 5/7)/2 < 0$.

It remains to verify \eqref{eq:Ooff}. When $k = \lceil L\delta \rceil \ge 3$, this is relatively easy: Here $2\delta^{-1} < L \le N/2$, so
\[
	L(N - L) \ > \ 2\delta^{-1}(N - 2\delta^{-1}) \ = \ 2 N \delta^{-1} \frac{n-3}{n-1}
	\ \ge \ N \delta^{-1} ,
\]
because $n \ge 5$. Hence,
\[
	g(L) < \frac{N}{8L(N - L)\Delta} \ < \ \frac{\delta}{8\Delta} .
\]
The case $k = 2$ is a bit more involved: Since
\[
	g(L)
	\ = \ \frac{\gamma(1 - \gamma) (2N\Delta + 1)}
		{(2L\Delta + 1 - \gamma)(2(N-L)\Delta + \gamma)} ,
\]
inequality \eqref{eq:Ooff} is equivalent to
\begin{equation}
\label{eq:Ooff2}
	7 \gamma(1 - \gamma) (2N\Delta^2 + \Delta)
	\ \le \ (2L\Delta + 1 - \gamma)(2(N-L)\Delta + \gamma) \delta .
\end{equation}
The left-hand side of \eqref{eq:Ooff2} equals
\[
	14 \gamma(1 - \gamma) N\Delta^2 + 7 \gamma(1 - \gamma) \Delta
	\ \le \ 14 \gamma(1 - \gamma) N \Delta^2 +2 \Delta ,
\]
because $7 \gamma(1 - \gamma) \le 7/4 < 2$, while the right-hand of \eqref{eq:Ooff2} side equals
\begin{align*}
	4 L& (N - L) \Delta^2 \delta + 2 ((1 - \gamma)(N-L) + \gamma L) \Delta \delta
		+ \gamma(1 - \gamma) \delta \\
	&\ge \ 4 L(N - L) \Delta^2 \delta + 2 L\delta \Delta
		\ > \ 4L(N - L) \Delta^2 \delta + 2 \Delta ,
\end{align*}
because $N - L \ge L$ and $L\delta > 1$. Consequently, it suffices to verify that
\begin{equation}
\label{eq:Ooff3}
	7 \gamma(1 - \gamma) N \ \le \ 2 L (N - L) \delta .
\end{equation}
To this end, note that $\gamma$ depends on $L$, namely, $\gamma = 2 - L\delta$, whence $L = (2 - \gamma) \delta^{-1}$ and
\[
	2 L(N - L)\delta \ = \ 2 (2 - \gamma) (N - (2 - \gamma) \delta^{-1})
	\ = \ 2 (2 - \gamma)(n - 1 - (2 - \gamma)) \delta^{-1} ,
\]
so \eqref{eq:Ooff3} is equivalent to
\begin{equation}
\label{eq:Ooff4}
	2 (2 - \gamma)(n-3 + \gamma) - 7 \gamma(1 - \gamma)(n-1) \ \ge \ 0 .
\end{equation}
But the left-hand side is
\begin{eqnarray*}
\lefteqn{4(n - 3) - 2\gamma (4.5 n - 8.5) + \gamma^2 (7n - 9) } \\
	&& \ge \ 4(n - 3) - \frac{(4.5n - 8.5)^2}{7n - 9} 
	      \ = \ \frac{4(n - 3)(7n - 9) - (4.5n - 8.5)^2}{7n - 9} .
\end{eqnarray*}
For $n \ge 5$, the denominator is strictly positive, and the derivative of the numerator is $15.5n - 43.5$, which is strictly positive, too. Thus it suffices to verify that the numerator is nonnegative for $n = 5$. Indeed, $4(n - 3)(7n - 9) - (4.5n - 8.5)^2 = 12$ for $n = 5$.

Finally, it follows from Bernoulli's inequality\footnote{$(1 + x)^m \ge 1 + mx$ for real numbers $x > -1$ and $m \ge 1$} 
that $(1 - 1/N)^{-(n-1)} \le (1 - (n-1)/N)^{-1}$. Now the inequalities for the total variation distance are an immediate consequence of Proposition~\ref{prop:rho.and.divergences}~(a) with $\psi(t) = (1 - t^{-1})_+$ and the fact that $Q(\{0\}) \le P(\{0\})$ and $Q(\{n\}) \le P(\{n\})$, whence
\[
	Q(\{g>f\})
	\ \le \ 1 - Q(\{0\}) - Q(\{n\})
	\ = \ 1 - \frac{[N-L]_n}{[N]_n} - \frac{[L]_n}{[N]_n} .
\]\\[-5ex]
\end{proof}

\begin{proof}[\bf Proof of Theorem~\ref{thm:BinPoiss1}]
Obviously, $\Lambda_n(0) = 0$. For $k \in \mathbb{N}_0$ we introduce the weights 
$b(k) = b_{n,p}(k) := \Bin(n,p)(\{k\})$ and $\pi(k) = \pi_{np}(k) := \Poiss(np)(\{k\}) = e^{-np} (np)^k / k!$. 
Obviously, $b(k) = 0$ for $k > n$, while for $0 \le k \le n$ and $p \in (0,1)$,
\[
	\lambda_{n,p}(k) \ := \ \log \frac{b(k)}{\pi(k)}
	\ = \ \log \Bigl( \frac{[n]_k}{n^k} \Bigr) + np + (n - k)\log(1 - p) .
\]
Note that the right hand side is a continuous function of $p \in [0,1)$ with limit 
$\lambda_{n,0}(k) := \log([n]_k/n^k) \le 0$ as $p \to 0$, where $\lambda_{n,0}(0) = 0$. Thus we may conclude that
\[
	\Lambda_n(p) \ = \ \max_{k = 0,1,\ldots,n} \lambda_{n,p}(k)
\]
is a continuous function of $p \in [0,1)$.

Next we need to determine the maximizer of $\lambda_{n,p}(\cdot)$. For $k \in \{0,1\ldots,n-1\}$,
\[
	\lambda_{n,p}(k+1) - \lambda_{n,p}(k)
	\ = \ \log(1 - k/n) - \log(1 - p)
	\ \begin{cases}
		\ge \ 0 & \text{if} \ k \le np , \\
		\le \ 0 & \text{if} \ k \ge np .
	\end{cases}
\]
Consequently,
\[
	\Lambda_n(p) \ = \ \lambda_{n,p}(\lceil np\rceil) .
\]

From now on we fix an integer $k \in \{1,\ldots,n\}$ and focus on 
$p \in \bigl[ (k-1)/n, k/n \bigr]$, so that $k = \lceil np\rceil$ if $p > (k-1)/n$. Then
\[
	\Lambda_n(p) \ = \ \log \Bigl( \frac{[n]_k}{n^k} \Bigr) + np + (n - k) \log(1 - p) .
\]
This is a concave function of $p$ with derivative
\[
	n - \frac{n-k}{1 - p}
	\ = \ \frac{k - np}{1 - p}
	\ \begin{cases}
		< \ 1/(1 - p) \\
		> \ 0
	\end{cases}
\]
if $(k-1)/n < p < k/n$. Since $1/(1 - p)$ is the derivative of $- \log(1 - p)$ with respect to $p$, and since $\Lambda_n(0) = 0 = - \log(1 - 0)$, this implies that
\[
	\Lambda_n(p) \ < \ - \log(1 - p)
	\quad\text{for} \ p \in (0,1) .
\]
On the other hand, $\Lambda_n$ is strictly increasing, whence
\[
	\Lambda_n(p) \ \le \ \Lambda_n(k/n) .
\]
But Corollary~\ref{cor:log.Gamma.increments} in Section~\ref{subsec:Log.Gamma} implies that
\[
	\log \Bigl( \frac{[n]_k}{n^k} \Bigr)
	\ = \ \log \Bigl( \frac{[n-1]_{k-1}}{n^{k-1}} \Bigr)
	\ = \ (n - k + 1/2) \log \Bigl( \frac{n}{n-k+1} \Bigr) + 1 - k + s_{k,n}
\]
with
\[
	- \frac{k-1}{12n(n-k+1)}
	\ < \ s_{k,n}
	\ < \ \min \Bigl( 0, - \frac{k-1}{12n(n-k+1)} + \frac{1}{12^2(n-k+1)^2} \Bigr) .
\]
Consequently,
\begin{align*}
	\Lambda_n(k/n) \
	&= \ \log \Bigl( \frac{[n]_k}{n^k} \Bigr) + k + (n - k)\log(1 - k/n) \\
	&\le \ (n - k + 1/2) \log \Bigl( \frac{n - k}{n-k+1} \Bigr) + 1
		- \frac{\log(1 - k/n)}{2} \\
	&< \ - \frac{\log(1 - k/n)}{2} ,
\end{align*}
where the last inequality follows from \eqref{ineq:log.ratio.2} with $x = n - k$, $a = 1$, and $b = 1/2$.

The refined bounds are for the quantity
\[
	D_n(p) \ := \ \Lambda_n(p) + \log(1 - p)/2 .
\]
For $p \in \bigl[(k-1)/n, k/n \bigr]$,
\[
	D_n(p) \ = \ \log \Bigl( \frac{[n]_k}{n^k} \Bigr) + np + (n - k + 1/2) \log(1 - p)
\]
and
\[
	D_n'(p) \ = \ n - \frac{n - k + 1/2}{1 - p}
	\ = \ \frac{k - 1/2 - np}{1 - p}
	\ \begin{cases}
		\ge \ 0 & \text{if} \ p \le (k - 1/2)/n , \\
		\le \ 0 & \text{if} \ p \ge (k - 1/2)/n .
	\end{cases}
\]
Consequently,
\begin{align*}
	D_n(p) \
	&\le \ D_n \Bigl( \frac{k-1/2}{n} \Bigr) \\
	&\le \ (n - k + 1/2) \log \Bigl( \frac{n - k + 1/2}{n-k+1} \Bigr)
		+ \frac{1}{2} - \frac{k-1}{12n(n-k+1)} + \frac{1}{12^2(n-k+1)^2} .
\end{align*}
It follows from \eqref{ineq:log.ratio.1} with $x = n - k+1/2$, $a = 1/2$ and $b = 0$ that
\begin{align*}
	(n - k + 1/2) \log \Bigl( \frac{n - k + 1/2}{n-k+1} \Bigr) + \frac{1}{2} \
	&= \ x \log \Bigl( \frac{x}{x + a} \Bigr) + a \\
	&< \ \frac{a^2}{2x + a} - \frac{2a^3 x}{3(2x + a)^3} \\
	&< \ \frac{1}{8(n - k) + 6} - \frac{n - k + 1/2}{12 \cdot 8 (n - k + 3/4)^3} ,
\end{align*}
and with $y := n - k + 3/4 \ge 3/4$,
\begin{align*}
	&\frac{n - k + 1/2}{12 \cdot 8 (n - k + 3/4)^3}
		\Big/ \frac{1}{12^2(n - k + 1)^2} \\
	&\ = \ \frac{3(y - 1/4)(y + 1/4)^2}{2 y^3}
		 \ > \ \frac{3(y^2 - 1/16)}{2y^2}
		 \ \ge \ \frac{4}{3} \ \ge \ 1 .
\end{align*}
Hence
\[
	D_n(p) \ \le \ \frac{1}{8(n - k) + 6} - \frac{k-1}{12n(n-k+1)} .
\]
On the other hand, the lower bound for $D_n(p)$ in \eqref{ineq:BinPoiss1} is trivial in case of $k = n$, and otherwise
\begin{align*}
	D_n(p) \
	&\ge \min_{j = k-1,k} D_n(j/n) \\
	&= \ \min_{j = k-1,k}
		\Bigl( (n - k + 1/2) \log \Bigl( \frac{n - j}{n - k + 1} \Bigr)
			+ 1 - k + j \Bigr) + s_{k,n} \\
	&> \ (n - k + 1/2) \log \Bigl( \frac{n - k}{n - k + 1} \Bigr) + 1
		- \frac{k-1}{12n(n-k+1)} \\
	&> \ - \frac{1}{12(n - k)(n - k + 1)} - \frac{k-1}{12n(n - k + 1)}
\end{align*}
by \eqref{ineq:log.ratio.3} with $x = n - k$ and $a = 1$.
\end{proof}

\begin{proof}[\bf Proof of Theorem~\ref{thm:Beta.Gamma}]
We start with the first statement of part~(ii). Let $\beta := \beta_{a,b}$ and $\gamma_c := \gamma_{a,c}$ for $c > 0$. Since $\beta(x) = 0$ for $x \ge 1$, it suffices to consider the log-density ratio
\[
	\lambda_c(x) := \log \frac{\beta}{\gamma_c}(x)
	\ = \ \log \frac{\Gamma(a+b)}{\Gamma(b)} - a \log c + (b - 1) \log(1 - x) + c x
\]
for $0 \le x < 1$, noting that the latter expression for $\lambda_c(x)$ is well-defined for all $x < 1$. The derivative of $\lambda_c$ equals
\[
	c - \frac{b - 1}{1 - x}
	\ = \ \frac{c}{1 - x} \Bigl( 1 - x - \frac{b-1}{c} \Bigr)
	\ = \ \frac{c}{1 - x} \Bigl( \frac{c-b+1}{c} - x \Bigr) ,
\]
and this is smaller or greater than zero if and only if $x$ is greater or smaller than the ratio $(c-b+1)/c$, respectively. This shows that in case of $c \le b-1$,
\begin{align*}
	\log \rho \bigl( \mathrm{Beta}(a,b), \mathrm{Gamma}(a,c) \bigr)
		\ = \ \lambda_c(0) \
	&= \ \log \frac{\Gamma(a+b)}{\Gamma(b)} - a \log c \\
	&\ge \ \log \frac{\Gamma(a+b)}{\Gamma(b)} - a \log(b-1) \\
	&= \ \log \rho \bigl( \mathrm{Beta}(a,b), \mathrm{Gamma}(a,b-1) \bigr) .
\end{align*}
For $c \ge b-1$,
\begin{align}
\nonumber
	\log \rho \bigl( & \mathrm{Beta}(a,b), \mathrm{Gamma}(a,c) \bigr)
		\ = \ \lambda_c \Bigl( \frac{c-b+1}{c} \Bigr) \\
\label{eq:Beta.Gamma}
	&= \ \log \frac{\Gamma(a+b)}{\Gamma(b)}
		- (a + b - 1) \log c + (b - 1) \log(b - 1)
		+ c - b + 1 .
\end{align}
But the derivative of the latter expression with respect to $c \ge b-1$ equals
\[
	1 - \frac{a+b-1}{c} ,
\]
so the unique minimizer of $\log \rho \bigl( \mathrm{Beta}(a,b), \mathrm{Gamma}(a,c) \bigr)$ with respect to $c > 0$ is $c = a+b-1$.

It remains to verify the inequalities
\begin{align}
\label{ineq:Beta.Gamma.1}
	\log \rho \bigl( \mathrm{Beta}(a,b), \mathrm{Gamma}(a,a+b) \bigr) \
	&\le \ - \frac{\log(1 - \delta)}{2} , \\
\label{ineq:Beta.Gamma.2}
	\log \rho \bigl( \mathrm{Beta}(a,b), \mathrm{Gamma}(a,a+b-1) \bigr) \
	&\le \ - \frac{\log(1 - \tilde{\delta})}{2} .
\end{align}
Then the total variation bounds of Theorem~\ref{thm:Beta.Gamma} follow from Proposition~\ref{prop:rho.and.divergences}~(a) and the elementary inequality \eqref{eq:sqrt.delta}. Corollary~\ref{cor:log.Gamma.increments} in Section~\ref{subsec:Log.Gamma} implies that
\begin{equation}
\label{ineq:Beta.Gamma.3}
	\log \frac{\Gamma(a+b)}{\Gamma(b)}
	\ < \ (a + b - 1/2) \log(a + b) - (b - 1/2) \log(b) - a .
\end{equation}
Combining this with \eqref{eq:Beta.Gamma} yields \eqref{ineq:Beta.Gamma.1}:
\begin{align*}
	\log \rho \bigl( &\mathrm{Beta}(a,b), \mathrm{Gamma}(a, a+b) \bigr) \\
	&= \ \log \frac{\Gamma(a+b)}{\Gamma(b)}
		- (a+b-1) \log(a+b) + (b-1) \log(b-1) + a + 1 \\
	&< \ \frac{\log(a + b)}{2} - \frac{\log(b - 1)}{2}
		+ 1 + (b - 1/2) \log \Bigl( \frac{b-1}{b} \Bigr) \\
	&= \ - \frac{\log(1 - \delta)}{2}
		+ 1 + (b - 1/2) \log \Bigl( \frac{b-1}{b} \Bigr) \\
	&< \ - \frac{\log(1 - \delta)}{2} ,
\end{align*}
by \eqref{ineq:log.ratio.2} with $(x,a,b) = (b-1,a,1/2)$. Concerning \eqref{ineq:Beta.Gamma.2}, if follows from \eqref{eq:Beta.Gamma} and \eqref{ineq:Beta.Gamma.3} that
\begin{align*}
	\log \rho \bigl( &\mathrm{Beta}(a,b), \mathrm{Gamma}(a, a+b-1) \bigr) \\
	&= \ \log \frac{\Gamma(a+b)}{\Gamma(b)}
		- (a+b-1) \log(a+b-1) + (b-1) \log(b-1) + a \\
	&< \ \frac{\log(a + b)}{2} - \frac{\log(b - 1)}{2}
		- (a + b - 1/2) \log \Bigl( \frac{a+b-1}{a+b} \Bigr)
		+ (b - 1/2) \log \Bigl( \frac{b-1}{b} \Bigr) \\
	&= \ - \frac{\log(1 - \tilde{\delta})}{2}
		+ \frac{1}{2} \Bigl( A \log \Bigl( \frac{1 - 1/A}{1 + 1/A} \Bigr)
			- B \log \Bigl( \frac{1 - 1/B}{1 + 1/B} \Bigr) \Bigr) ,
\end{align*}
where $A := 2b - 1$ and $B := 2(a + b) - 1$. Now \eqref{ineq:Beta.Gamma.2} follows from
\[
	A \log \Bigl( \frac{1 - 1/A}{1 + 1/A} \Bigr)
		- B \log \Bigl( \frac{1 - 1/B}{1 + 1/B} \Bigr)
	\ = \ \sum_{\ell=0}^\infty \frac{B^{-2\ell} - A^{-2\ell}}{2\ell + 1}
	\ < \ 0 ,
\]
because $A < B$.

In the special case of $a = 1$, we do not need \eqref{ineq:Beta.Gamma.3} but get via \eqref{eq:Beta.Gamma} the explicit expression
\begin{align*}
	\log \rho \bigl( \mathrm{Beta}(1,b), \mathrm{Gamma}(1, b) \bigr) \
	&= \ \log \frac{\Gamma(b+1)}{\Gamma(b)}
		- b \log(b) + (b-1) \log(b-1) + 1 \\
	&= \ (b - 1) \log(1 - 1/b) + 1 ,
\end{align*}
because $\Gamma(b+1) = b \Gamma(b)$. Now the standard Taylor series for $\log(1 - x)$ yields that
\begin{align*}
	\log \rho \bigl(& \mathrm{Beta}(1,b), \mathrm{Gamma}(1, b) \bigr) \\
	&= \ - (b - 1) \sum_{\ell=1}^\infty \frac{b^{-\ell}}{\ell} + 1
		\ = \ \sum_{\ell=1}^\infty
			\Bigl( \frac{b^{-\ell}}{\ell} - \frac{b^{-\ell}}{\ell+1} \Bigr)
		\ = \ \sum_{\ell=1}^\infty \frac{b^{-\ell}}{\ell(\ell+1)} \\
	&< \ \frac{1}{2b} + \frac{1}{6b^2} + \frac{1}{12 b^3} \sum_{j=0}^\infty b^{-j}
		\ = \ \frac{1}{2b} + \frac{1}{6b^2} + \frac{1}{12 b^2(b - 1)} ,
\end{align*}
and in case of $b \ge 2$, the latter expression is not larger than
\[
	\frac{1}{2b} + \frac{1}{6b^2} + \frac{1}{12b^2}
		\ = \ \frac{1}{2b} + \frac{1}{4b^2} .
\]\\[-5ex]
\end{proof}

\begin{proof}[\bf Proof of Lemma~\ref{lem:N01.student}]
By Proposition~\ref{prop:rho.and.divergences}~(a) and the inequality $1 - \exp(-x) \le x$ for $x \ge 0$, it suffices to verify the claims about $\log \rho \bigl( N(0,1), t_r \bigr)$. Note first that
\[
	\log \frac{\phi(x)}{f_r(x)}
	\ = \ \log \frac{\Gamma(r/2) \sqrt{r/2}}{\Gamma((r+1)/2)}
		+ \frac{r+1}{2} \log \Bigl( 1 + \frac{x^2}{r} \Bigr) - \frac{x^2}{2}
\]
and
\[
	\frac{\partial}{\partial (x^2)} \, \log \frac{\phi(x)}{f_r(x)}
	\ = \ \frac{r+1}{2(r + x^2)} - \frac{1}{2}
	\ = \ \frac{1 - x^2}{2(r + x^2)} ,
\]
whence
\[
	\log \rho \bigl( N(0,1), t_r \bigr)
	\ = \ \log \frac{\Gamma(r/2) \sqrt{r/2}}{\Gamma((r+1)/2)}
		- \frac{1}{2} + \frac{r+1}{2} \log \Bigl( 1 + \frac{1}{r} \Bigr) .
\]
On the one hand, the Taylor expansion $- \log(1 - x) = \sum_{k=1}^\infty x^k/k$ yields that
\begin{align*}
	- \frac{1}{2} + \frac{r+1}{2} \log \Bigl( 1 + \frac{1}{r} \Bigr) \
	&= \ - \frac{1}{2}
		- \frac{r+1}{2} \log \Bigl( \frac{r}{r+1} \Bigr) 
	      \ = \ - \frac{1}{2} + \frac{r+1}{2} \sum_{k=1}^\infty \frac{1}{k (r+1)^k} \\ 
	&= \ \frac{1}{2} \sum_{k=2}^\infty \frac{1}{k (r+1)^{k-1}} ,
\end{align*}
and the latter series equals
\begin{eqnarray*}
\lefteqn{\frac{1}{4(r+1)}
	+ \frac{1}{2(r+1)^2} \sum_{\ell=0}^\infty \frac{1}{(\ell+3) (r+1)^\ell }} \\
	& < & \ \frac{1}{4(r+1)} + \frac{1}{6(r+1)^2} \sum_{\ell=0}^\infty (r+1)^{-\ell}  
	         \  =  \ \frac{1}{4(r+1)} + \frac{1}{6(r+1)^2 (1 - (r+1)^{-1})}\\
	& =  &\ \frac{1}{4(r+1)} + \frac{1}{6(r+1) r} 
	         \ = \ \frac{1}{4r} - \frac{1}{4r(r+1)} + \frac{1}{6(r+1) r} 
	         \ =  \ \frac{1}{4r} - \frac{1}{12r(r+1)} .
\end{eqnarray*}
Moreover, it follows from Lemma~\ref{lem:Ex.sqrt.Gamma} in Section~\ref{subsec:Log.Gamma} with $x := r/2$ that
\begin{align*}
	\log \frac{\Gamma(r/2) \sqrt{r/2}}{\Gamma((r+1)/2)}
	\ < \ \frac{1}{4r} + \frac{1}{12 r(r^2 - 1)}
	\ &= \ \frac{1}{4r} + \frac{1}{12 r(r+1)(r-1)} \\
	&\le \ \frac{1}{4r} + \frac{1}{12 r(r+1)} ,
\end{align*}
because $r-1 \ge 1$ by assumption. Consequently,
\[
	\log \rho \bigl( N(0,1), t_r \bigr)
	\ < \ \frac{1}{2r} .
\]
On the other hand, the previous considerations and Lemma~\ref{lem:Ex.sqrt.Gamma} imply that
\[
	- \frac{1}{2} + \frac{r+1}{2} \log \Bigl( 1 + \frac{1}{r} \Bigr)
	\ > \ \frac{1}{4(r+1)}
\]
and
\[
	\log \frac{\Gamma(r/2) \sqrt{r/2}}{\Gamma((r+1)/2)}
	\ > \ \frac{1}{4(r+1)} ,
\]
whence
\[
	\log \rho(N(0,1), t_r) \ > \ \frac{1}{2(r+1)} .
\]\\[-5ex]
\end{proof}

\subsection{Auxiliary Results for the Gamma Function}
\label{subsec:Log.Gamma}

In what follows, let
\[
	h(x) \ := \ \log \Gamma(x)
	\ = \ \log \int_0^\infty t^{x-1} e^{-t} \, dt ,
	\quad x > 0 .
\]
With a random variable $Y_x \sim \mathrm{Gamma}(x,1)$ one may write
\[
	h'(x) \ = \ \Ex (\log Y_x)
	\quad\text{and}\quad
	h''(x) \ = \ \Var(\log Y_x) .
\]
The functions $h'$ and $h''$ are known as the digamma and trigamma functions; see e.g.,  \nocite{NIST_2010}{Olver et al.\ (2010)}, Section~5.15. This shows that $h(x)$ is strictly convex in $x > 0$. Moreover, it follows from concavity of $\log(\cdot)$ and Jensen's inequality that
\[
	h'(x) \ < \ \log \Ex(Y_x) \ = \ \log x .
\]
The well-known identity $\Gamma(x+1) = x \Gamma(x)$ is equivalent to
\[
	h(x+1) - h(x) \ = \ \log x .
\]

\paragraph{Binet's first formula and Stirling's approximation.}
Binet's first integral formula states that
\begin{equation}
\label{eq:Binet}
	h(x) \ = \ \log \sqrt{2\pi} + (x - 1/2) \log x - x + R(x) ,
\end{equation}
where
\[
	R(x) \ := \ \int_0^\infty e^{-tx} w(t) \, dt
	\quad\text{and}\quad
	w(t) \ := \ \frac{1}{t} \Bigl( \frac{1}{2} - \frac{1}{t} + \frac{1}{e^t-1} \Bigr) ,
\]
see Chapter~12 of \cite{Whittaker_Watson_1996}. The following lemma provides a lower and upper bound for $w(t)$, and these yield rather precise bounds for the remainder $R(x)$.

\begin{Lemma}
\label{lem:Binet.Stirling}
For arbitrary $t > 0$,
\[
	12^{-1} e^{-t/12} \ < \ w(t) \ < \ 12^{-1} .
\]
In particular, the remainder $R(x)$ in Binet's formula \eqref{eq:Binet} is strictly decreasing in $x > 0$ and satisfies
\[
	(12x + 1)^{-1} \ < \ R(x) \ < \ (12x)^{-1} .
\]
\end{Lemma}

Since $n! = \Gamma(n+1)$, Lemma~\ref{lem:Binet.Stirling} implies a slight improvement of the Stirling approximation by \cite{Robbins_1955}: For arbitrary integers $n \ge 0$,
\begin{equation}
\label{eq:Stirling.Robbins.new}
	\log(n!) \ = \ \log \sqrt{2\pi} + (n + 1/2) \log(n + 1) - n - 1 + s_n
\end{equation}
with
\[
	\frac{1}{12(n+1) + 1} \ < \ s_n \ < \ \frac{1}{12(n+1)} .
\]
In addition, Binet's formula \eqref{eq:Binet} and Lemma~\ref{lem:Binet.Stirling} lead to useful inequalities for the increments of $h(\cdot)$.

\begin{Corollary}
\label{cor:log.Gamma.increments}
For arbitrary $0 < a < b$,
\[
	h(b) - h(a) \ = \ (b - 1/2) \log(b) - (a - 1/2) \log(a) - (b - a) + s(a,b)
\]
where
\[
	- \frac{b-a}{12ab}
		\ < \ s(a,b)
			\ < \ \min \Bigl( 0, - \frac{b-1}{12ab} + \frac{1}{12^2 a^2} \Bigr)
\]
\end{Corollary}

\begin{proof}[\bf Proof of Lemma~\ref{lem:Binet.Stirling}]
The series expansion of the exponential function and some elementary algebra lead to the representation
\begin{align*}
	w(t) \
	&= \ \frac{1}{t} \Bigl( \frac{1}{2} - \frac{1}{t} + \frac{1}{e^t - 1} \Bigr) \\
	&= \ \frac{t(e^t - 1) - 2(e^t - 1 - t)}{2t^3} \Big/ \frac{e^t - 1}{t} \\
	&= \ \sum_{m=1}^\infty \frac{a_m t^{m-1}}{m!}
		\Big/ \sum_{m=1}^\infty \frac{t^{m-1}}{m!} ,
\end{align*}
with
\[
	a_m \ := \ \frac{m}{2(m+1)(m+2)} .
\]
Note that $a_1 = 12^{-1}$, and
\[
	\frac{a_{m+1}}{a_m}
	\ = \ \frac{(m+1)^2}{(m+1)^2 + m - 1}
	\ \begin{cases}
		= \ 1 & \text{for} \ m = 1 , \\
		< \ 1 &\text{for} \ m \ge 2 .
	\end{cases}
\]
This shows that $a_m \le 12^{-1}$ with strict inequality for $m \ge 3$. Consequently, $w(t) < 12^{-1}$.

The reverse inequality, $w(t) > 12^{-1} e^{-t/12}$, is equivalent to
\[
	12 \, \frac{t(e^t - 1) - 2(e^t - 1 - t)}{2t^3}
	\ > \ \frac{e^t - 1}{t} e^{-t/12} .
\]
The left hand side equals $12 \sum_{m=1}^\infty a_m t^{m-1}/m!$, while the right hand side equals
\begin{align*}
	\frac{e^{(11/12)t} - e^{-t/12}}{t} \
	&= \ \sum_{m=1}^\infty \frac{\bigl( (11/12)^m - (-1/12)^m \bigr) t^{m-1}}{m!} \\
	&< \ \sum_{m=1}^\infty \frac{c_m t^{m-1}}{m!}
		\quad\text{with} \ c_m := (11/12)^m + (1/12)^m .
\end{align*}
Note that $12 a_1 = 1 = c_1$. Consequently, $w(t) > 12^{-1} e^{-t/12}$ for all $t > 0$, provided that $12 a_m \ge c_m$ for all $m \ge 2$. But $c_2 = 61/72$ and $c_{m+1}/c_m < 11/12$, whence $c_m \le (61/72) (11/12)^{m-2}$ for $m \ge 2$. Consequently, it suffices to show that
\[
	12 (12/11)^{m-2} a_m \ \ge \ 61/72 \quad\text{for all} \ m \ge 2 .
\]
But
\[
	\frac{(12/11)^{m+1 - 2} a_{m+1}}{(12/11)^{m-2} a_m}
	\ = \ \frac{1 + 1/11}{1 + (m-1)/(m+1)^2}
	\ < \ 1
\]
if and only if $m^2 - 9m < -12$, and for integers $m \ge 2$ this is equivalent to $m \le 7$. Hence
\[
	\min_{m \ge 2} \, 12 (12/11)^{m-2} a_m \ = \ 12 (12/11)^{8-2} a_8 \ \ge \ 0.89
	\ > \ 0.85 \ \ge \ 61/72 .
\]

Since for any fixed $t > 0$, the integrand $e^{-tx} w(t)$ is strictly decreasing in $x > 0$, the remainder $R(x)$ is strictly decreasing in $x > 0$. The two bounds for $w(t)$ imply that $R(x)$ is larger than $12^{-1} \int_0^\infty e^{-t(x + 1/12)} \, dt = (12x + 1)^{-1}$ and smaller than $12^{-1} \int_0^\infty e^{-tx} \, dt = (12 x)^{-1}$.
\end{proof}

\begin{proof}[\bf Proof of Corollary~\ref{cor:log.Gamma.increments}]
Writing $h(x) = \log \sqrt{2\pi} + \tilde{h}(x) + R(x)$ with the auxiliary function $\tilde{h}(x) := (x - 1/2) \log x - x$, the remainder term $s(a,b)$ equals $R(b) - R(a)$. But
\[
	R(a) - R(b) \ = \ \int_0^\infty (e^{-ta} - e^{-tb}) w(t) \, dt ,
\]
and since $e^{-ta} - e^{-tb} > 0$, it follows from $0 < w(t) < 12^{-1}$ that
\[
	0 \ < \ R(a) - R(b) \ < \ \frac{1}{12} \int_0^\infty (e^{-ta} - e^{-tb}) \, dt
	\ = \ \frac{1}{12a} + \frac{1}{12 b} \ = \ \frac{b-a}{12ab} .
\]
Moreover, since $w(t) > 12^{-1} e^{-t/12}$,
\begin{align*}
	R(a) - R(b) \
	&> \ \frac{1}{12} \int_0^\infty (e^{-t(a+1/12)} - e^{-t(b+1/12)}) \, dt
		\ = \ \frac{1}{12a + 1} - \frac{1}{12b+1} \\
	&= \ \frac{b-a}{12ab} + \frac{1}{12b(12b+1)} - \frac{1}{12a(12a+1)}
		\ > \ \frac{b-a}{12ab} - \frac{1}{12^2 a^2} .
\end{align*}\\[-5ex]
\end{proof}

\paragraph{Special increments of $h$.}
In connection with student distributions, we need lower and upper bounds for the quantities $h(x + 1/2) - h(x) - \log(x)/2$. With a random variable $Y_x \sim \mathrm{Gamma}(x,1)$, the latter expression equals $\log \Ex \sqrt{Y_x/x}$, so it follows from Jensen's inequality that $h(x+1/2) - h(x) - \log(x)/2 < \log \sqrt{\Ex(Y_x/x)} = 0$. The next lemma shows that $h(x + 1/2) - h(x) - \log(x)/2$ is close to to $- 1/(8x)$ for large $x$.

\begin{Lemma}
\label{lem:Ex.sqrt.Gamma}
For arbitrary $x > 0$,
\[
	- \frac{1}{8x} - \frac{1}{24 x (4x^2 - 1)_+}
	\ < \ h(x + 1/2) - h(x) - \frac{\log x}{2}
	\ < \ - \frac{1}{8(x + 1/2)} .
\]
\end{Lemma}

\begin{proof}[\bf Proof of Lemma~\ref{lem:Ex.sqrt.Gamma}]
Let us first mention that the second derivative of the log-gamma function $h$ is given by Gauss' formula
\[
	h''(x) \ = \ \sum_{n=0}^\infty \frac{1}{(x+n)^2} ,
\]
see Chapter~12 of \cite{Whittaker_Watson_1996}. In particular, $h''$ is strictly convex and decreasing on $(0,\infty)$ with
\[
	h''(x) \ > \ \int_x^\infty \frac{1}{y^2} \, dy \ = \ \frac{1}{x} ,
\]
because $(x+n)^{-2} > \int_{x+n}^{x+n+1} y^{-2} \, dy$.

Now we start with a general consideration about second order differences of $h$: For arbitrary $0 < a < z$,
\begin{align*}
	h(z+a) + h(z-a) - 2 h(z) \
	&= \ \bigl( h(z+a) - h(z) \bigr) - \bigl(h(z) - h(z-a)  \bigr) \\
	&= \ \int_0^a \bigl( h'(z+u) - h'(z-a+u) \bigr) \, du \\
	&= \ \int_0^a \int_0^a h''(z-a+u+v) \, dv \, du \\
	&= \ a^2 \Ex h''(z-a + a(U+V)) ,
\end{align*}
where $U$ and $V$ are independent random variables with uniform distribution on $[0,1]$. Since $h''$ is convex and $h''(z) > 1/z$, it follows from Jensen's inequality that
\[
	h(z+a) + h(z-a) - 2 h(z)
	\ \ge \ a^2 h''(z-a + a \Ex(U+V))
	\ = \ a^2 h''(z) \ > \ \frac{a^2}{z} .
\]
Note also that the distribution of $W := U+V$ is given by the triangular density $f(w) := (1 - |w-1|)_+$, so
\begin{align*}
	h(z+a) + h(z-a) - 2 h(z) \
	&= \ a^2 \int_{\R} (1 - |w-1|)_+ h''(z-a + aw) \, dw \\
	&= \ \int_{\R} (a - |a(w-1)|)_+ h''(z + a(w-1)) \, a \, dw \\
	&= \ \int_{\R} (a - |t|)_+ h''(z+t) \, dt .
\end{align*}

We first apply these findings with $z = x + 1/2$ and $a = 1/2$: Since $h(x+1) - h(x) = \log x$,
\begin{align*}
	\frac{\log x}{2} - \bigl( h(x + 1/2) - h(x) \bigr) \
	&= \ \frac{h(x+1) - h(x)}{2} - h(x+1/2) + h(x) \\
	&= \ \frac{1}{2} \bigl( h(x+1) + h(x) - 2 h(x + 1/2) \bigr) \\
	&\ge \ \frac{1}{8(x + 1/2)} ,
\end{align*}
which gives us the upper bound for $h(x+1/2) - h(x) - \log(x)/2$. Furthermore,
\[
	\frac{\log x}{2} - \bigl( h(x + 1/2) - h(x) \bigr) \
	\ = \ \frac{1}{2} \int_{\R} (1/2 - |t|)_+ h''(x + 1/2 + t) \, dt .
\]
On the other hand, if $x > 1/2$, then with $z = x+1/2$ and $a = 1$ we obtain
\begin{align*}
	\log \Bigl( \frac{x + 1/2}{x - 1/2} \bigr) \
	&= \ \bigl( h(x + 3/2) - h(x + 1/2) \bigr)
		- \bigl( h(x+1/2) - h(x-1/2) \bigr) \\
	&= \ \int_{\R} (1 - |t|)_+ h''(x + 1/2 + t) \, dt .
\end{align*}
Note that
\[
	\Delta(t) \ := \ \frac{1}{8} (1 - |t|)_+ - \frac{1}{2} (1/2 - |t|)_+
\]
has the following properties:
\[
	\int_{\R} \Delta(t) \, dt \ = \ \int_{\R} \Delta(t) t \, dt \ = \ 0
\]
and
\[
	\Delta(t) \ \begin{cases}
		< 0 & \text{if} \ |t| < 1/3 , \\
		\ge 0 & \text{if} \ |t| \ge 1/3 .
	\end{cases}
\]
These properties plus the convexity of $h''$ imply that
\[
	\int_{\R} \Delta(t) h''(x + 1/2 + t) \, dt \ \ge \ 0 .
\]
Indeed, the latter integral doesn't change if we replace $h''(x + 1/2 + t)$ with $g(t) := h''(x + 1/2 + t) + a + bt$ with constants $a, b$ such that $g(\pm 1/3) = 0$. But then, by convexity of $g$ and the sign changes of $\Delta$, we have that $g \Delta \ge 0$. Consequently,
\begin{align*}
	\frac{\log x}{2} - \bigl( h(x + 1/2) - h(x) \bigr) \
	&= \ \frac{1}{2} \int_{\R} (1/2 - |t|)_+ h''(x + 1/2 + t) \, dt \\
	&\le \ \frac{1}{8} \int_{\R} (1 - |t|)_+ h''(x + 1/2 + t) \, dt \\
	&= \ \frac{1}{8} \log \Bigl( \frac{x + 1/2}{x - 1/2} \Bigr) .
\end{align*}
Finally, with $y := (2x)^{-1} < 1$, the latter expression equals
\begin{align*}
	\frac{1}{8} \log \Bigl( \frac{1 + y}{1 - y} \Bigr)
		\ = \ \frac{1}{4} \sum_{\ell=0}^\infty \frac{y^{2\ell+1}}{2\ell + 1}
	&= \ \frac{y}{4}
		+ \frac{1}{4} \sum_{\ell=1}^\infty \frac{y^{2\ell + 1}}{2\ell + 1} \\
	&< \ \frac{y}{4} + \frac{y^3}{12 (1 - y^2)} \\
	&= \ \frac{1}{8x} + \frac{1}{24x (4x^2 - 1)} .
\end{align*}\\[-7ex]
\end{proof}

\paragraph{Acknowledgements.}
Constructive comments of David Ginsbourger, Dominic Schuhmacher and Kaspar Stucki on an 
early version of this paper are gratefully acknowledged. We also thank Lutz Mattner for drawing our attention to the technical report \cite{Christensen_Fischer_Kvols_1995} and for pointing out the connection between the ratio measure and the mixture index of fit. Constructive comments of a referee led to further improvements such as Proposition~\ref{prop:rho.and.divergences}.

\bibliographystyle{ims}
\bibliography{BinPoissApprox}

\begin{thebibliography}{21}
\expandafter\ifx\csname natexlab\endcsname\relax\def\natexlab#1{#1}\fi
\expandafter\ifx\csname url\endcsname\relax
  \def\url#1{\texttt{#1}}\fi
\expandafter\ifx\csname urlprefix\endcsname\relax\def\urlprefix{URL }\fi

\bibitem[{Antonelli and Regoli(2005)}]{Antonelli_Regoli_2005}
\textsc{Antonelli, S.} and \textsc{Regoli, G.} (2005).
\newblock On the {P}oisson-binomial relative error.
\newblock \textit{Statist. Probab. Lett.} \textbf{71} 249--256.

\bibitem[{Barbour and Hall(1984)}]{Barbour_Hall_1984}
\textsc{Barbour, A.~D.} and \textsc{Hall, P.} (1984).
\newblock On the rate of {P}oisson convergence.
\newblock \textit{Math. Proc. Cambridge Philos. Soc.} \textbf{95} 473--480.

\bibitem[{Christensen et~al.(1995)Christensen, Fischer and
  Kvols}]{Christensen_Fischer_Kvols_1995}
\textsc{Christensen, J.}, \textsc{Fischer, P.} and \textsc{Kvols, K.} (1995).
\newblock On the ratio of binomial and poisson probability distributions.
\newblock Tech. Rep.~7, Matematisk Institut, Kobenhavns Universitet.

\bibitem[{Diaconis and Freedman(1980)}]{Diaconis_Freedman_1980}
\textsc{Diaconis, P.} and \textsc{Freedman, D.} (1980).
\newblock Finite exchangeable sequences.
\newblock \textit{Ann. Probab.} \textbf{8} 745--764.

\bibitem[{Diaconis and Freedman(1987)}]{Diaconis_Freedman_1987}
\textsc{Diaconis, P.} and \textsc{Freedman, D.} (1987).
\newblock A dozen de {F}inetti-style results in search of a theory.
\newblock \textit{Ann. Inst. H. Poincar\'{e} Probab. Statist.} \textbf{23}
  397--423.

\bibitem[{D{\"u}mbgen and Wellner(2020)}]{Duembgen_Wellner_2019}
\textsc{D{\"u}mbgen, L.} and \textsc{Wellner, J.~A.} (2020).
\newblock The density ratio of {P}oisson binomial versus {P}oisson
  distributions.
\newblock \textit{Statist. Probab. Lett.} \textbf{165} 108862.
\newblock (arXiv:1910.03444).

\bibitem[{Ehm(1991)}]{Ehm_1991}
\textsc{Ehm, W.} (1991).
\newblock Binomial approximation to the {P}oisson binomial distribution.
\newblock \textit{Statist. Probab. Lett.} \textbf{11} 7--16.

\bibitem[{Freedman(1977)}]{Freedman_1977}
\textsc{Freedman, D.} (1977).
\newblock A remark on the difference between sampling with and without
  replacement.
\newblock \textit{J. Amer. Statist. Assoc.} \textbf{72} 681.

\bibitem[{Hall(1979)}]{Hall_1979}
\textsc{Hall, P.} (1979).
\newblock On the rate of convergence of normal extremes.
\newblock \textit{J. Appl. Probab.} \textbf{16} 433--439.

\bibitem[{Hall and Wellner(1979)}]{Hall_Wellner_1979}
\textsc{Hall, W.~J.} and \textsc{Wellner, J.~A.} (1979).
\newblock The rate of convergence in law of the maximum of an exponential
  sample.
\newblock \textit{Statist. Neerlandica} \textbf{33} 151--154.

\bibitem[{Hodges and Le~Cam(1960)}]{Hodges_LeCam_1960}
\textsc{Hodges, J.~L., Jr.} and \textsc{Le~Cam, L.} (1960).
\newblock The {P}oisson approximation to the {P}oisson binomial distribution.
\newblock \textit{Ann. Math. Statist.} \textbf{31} 737--740.

\bibitem[{Holmes(2004)}]{Holmes_2004}
\textsc{Holmes, S.} (2004).
\newblock Stein's method for birth and death chains.
\newblock In \textit{Stein's method: expository lectures and applications},
  vol.~46 of \textit{IMS Lecture Notes Monogr. Ser.} Inst. Math. Statist.,
  Beachwood, OH, 45--67.

\bibitem[{Olver et~al.(2010)Olver, Lozier, Boisvert and Clark}]{NIST_2010}
\textsc{Olver, F. W.~J.}, \textsc{Lozier, D.~W.}, \textsc{Boisvert, R.~F.} and
  \textsc{Clark, C.~W.} (eds.)  (2010).
\newblock \textit{N{IST} handbook of mathematical functions}.
\newblock U.S. Department of Commerce, National Institute of Standards and
  Technology, Washington, DC; Cambridge University Press, Cambridge.

\bibitem[{Pinelis(2015)}]{Pinelis_2015}
\textsc{Pinelis, I.} (2015).
\newblock Exact bounds on the closeness between the {S}tudent and standard
  normal distributions.
\newblock \textit{ESAIM Probab. Stat.} \textbf{19} 24--27.

\bibitem[{Reiss(1993)}]{Reiss_1993}
\textsc{Reiss, R.-D.} (1993).
\newblock \textit{A course on point processes}.
\newblock Springer Series in Statistics, Springer-Verlag, New York.

\bibitem[{Robbins(1955)}]{Robbins_1955}
\textsc{Robbins, H.} (1955).
\newblock A remark on {S}tirling's formula.
\newblock \textit{Amer. Math. Monthly} \textbf{62} 26--29.

\bibitem[{Rudas et~al.(1994)Rudas, Clogg and
  Lindsay}]{Rudas_Clogg_Lindsay_1994}
\textsc{Rudas, T.}, \textsc{Clogg, C.~C.} and \textsc{Lindsay, B.~G.} (1994).
\newblock A new index of fit based on mixture methods for the analysis of
  contingency tables.
\newblock \textit{J. Roy. Statist. Soc. Ser. B} \textbf{56} 623--639.

\bibitem[{Runnenburg and Vervaat(1969)}]{Runnenburg_Vervaat_1969}
\textsc{Runnenburg, J.~T.} and \textsc{Vervaat, W.} (1969).
\newblock Asymptotical independence of the lengths of subintervals of a
  randomly partitioned interval; a sample from {S}. {I}keda's work.
\newblock \textit{Statistica Neerlandica} \textbf{23} 67--77.

\bibitem[{Serfling(1975)}]{Serfling_1975}
\textsc{Serfling, R.~J.} (1975).
\newblock A general {P}oisson approximation theorem.
\newblock \textit{Ann. Probability} \textbf{3} 726--731.

\bibitem[{von Neumann(1951)}]{vonNeumann_1951}
\textsc{von Neumann, J.} (1951).
\newblock Various techniques used in connection with random digits. monte carlo
  methods.
\newblock \textit{Nat. Bureau Standards} \textbf{12} 36--38.

\bibitem[{Whittaker and Watson(1996)}]{Whittaker_Watson_1996}
\textsc{Whittaker, E.~T.} and \textsc{Watson, G.~N.} (1996).
\newblock \textit{A course of modern analysis}.
\newblock Cambridge Mathematical Library, Cambridge University Press,
  Cambridge.

\end{thebibliography}

\end{document}